\documentclass[11pt,reqno]{amsart}%
\usepackage{graphicx,adjustbox}
\usepackage{stmaryrd}
\usepackage{mathtools, amsfonts, amssymb, amsthm}
\usepackage{enumerate, url}
\usepackage[margin=1in]{geometry}
\usepackage{caption}
\usepackage{algorithm}
\usepackage{algpseudocode}
\usepackage{mleftright}
\usepackage[normalem]{ulem}
\usepackage{hyperref}
\usepackage{tabu}
\usepackage{tikz-cd}

\usepackage{xcolor}
\hypersetup{
    colorlinks,
    linkcolor={red!50!black},
    citecolor={blue!50!black},
    urlcolor={blue!80!black}
}

\let\latexcirc=\circ
\newcommand{\ccirc}{\mathbin{\mathchoice
  {\xcirc\scriptstyle}
  {\xcirc\scriptstyle}
  {\xcirc\scriptscriptstyle}
  {\xcirc\scriptscriptstyle}
}}
\newcommand{\xcirc}[1]{\vcenter{\hbox{$#1\latexcirc$}}}
\let\circ\ccirc

\makeatletter
\def\Ddots{\mathinner{\mkern1mu\raise\p@
\vbox{\kern7\p@\hbox{.}}\mkern2mu
\raise4\p@\hbox{.}\mkern2mu\raise7\p@\hbox{.}\mkern1mu}}
\makeatother

\newtheorem{theorem}{Theorem}[section]
\newtheorem{proposition}[theorem]{Proposition}
\newtheorem{proposition/definition}[theorem]{Proposition/Definition}
\newtheorem{lemma}[theorem]{Lemma}
\newtheorem{corollary}[theorem]{Corollary}

\theoremstyle{definition}

\theoremstyle{remark}

\newcommand{\tp}{{\scriptscriptstyle\mathsf{T}}}
\newcommand{\F}{{\scriptscriptstyle\mathsf{F}}}

\let\O\undefined
\let\S\undefined

\DeclareMathOperator{\O}{O}
\DeclareMathOperator{\S}{S}
\DeclareMathOperator{\I}{I}
\DeclareMathOperator{\T}{\mathbb{T}}

\DeclareMathOperator{\J}{J}

\DeclareMathOperator{\Gr}{Gr}

\DeclareMathOperator{\Flag}{Flag}

\DeclareMathOperator{\tr}{tr}
\DeclareMathOperator{\proj}{proj}
\DeclareMathOperator{\im}{im}
\DeclareMathOperator{\rank}{rank}
\DeclareMathOperator{\diag}{diag}

\DeclareMathOperator*{\argmax}{argmax}

\DeclareMathOperator{\vect}{vec}

\newcommand{\Rmnum}[1]{\expandafter\@slowromancap\Romannumeral #1@} 
\begin{document}
\title{Simpler flag optimization}
\author[Z.~Lai]{Zehua Lai}
\address{Computational and Applied Mathematics Initiative,
University of Chicago, Chicago, IL 60637-1514.}
\email{laizehua@uchicago.edu, lekheng@uchicago.edu}
\author[L.-H.~Lim]{Lek-Heng~Lim}
\author[K.~Ye]{Ke~Ye}
\address{KLMM, Academy of Mathematics and Systems Science, Chinese Academy of Sciences,
Beijing 100190, China}
\email{keyk@amss.ac.cn}
\begin{abstract}
We study the geometry of flag manifolds under different embeddings into a product of Grassmannians. We show that differential geometric objects and operations — tangent
vector, metric, normal vector, exponential map, geodesic, parallel transport, gradient, Hessian,
etc — have closed-form analytic expressions that are computable with standard numerical linear
algebra. Furthermore, we are able to derive a coordinate descent method in the flag manifold that performs well compared to other gradient descent methods.
\end{abstract}
\maketitle
\tableofcontents
\addtocontents{toc}{\protect\setcounter{tocdepth}{1}}
\section{Introduction}\label{sec:intro}
Let $d \le n$ be positive integers and let $(n_1, \dots, n_d)$ be a sequence integers such that $0 < n_1 < \cdots < n_d < n$. We denote by $\Flag(n_1,\dots,n_d;\mathbb{R}^n)$ the set of all flags in $\mathbb{R}^n$ of type $(n_1,\dots, n_d)$:
\[
\Flag(n_1,\dots,n_d;n) \coloneqq \left\lbrace \left\lbrace \mathbb{V}_k \right\rbrace_{k=1}^d:\mathbb{V}_k \subsetneq \mathbb{V}_{k+1} \subsetneq \mathbb{R}^n, \dim \mathbb{V}_k = n_k,k = 1,\dots, d-1
\right\rbrace.
\] 
\section{Preliminaries}
\subsection{differential geometry of Grassmann manifolds}
Let $k \le n$ be positive integer. We denote by  $\Gr(k,n)$ the Grassmann manifold of $k$ dimensional subspaces of $\mathbb{R}^n$. According to \cite{LLY20}, $\Gr(k,n)$ can be characterized as a submanifold of $\O(n) \cap \S_n$, i.e., 
\begin{align}
\Gr(k,n) &\simeq \left\lbrace 
Q\in \O(n) \cap \S_n: \tr (Q) = 2k - n
\right\rbrace \label{eq:modelGr1} \\
&= \left\lbrace 
V \I_{k,n-k} V^\tp: V\in \O(n)
\right\rbrace. \label{eq:modelGr2}
\end{align}
Here $\S_n$ is the space of $n\times n$ symmetric matrices and $\I_{k,n-k}$ is the diagonal matrix $\begin{bmatrix}
\I_k & 0 \\
0 & -\I_{n-k}
\end{bmatrix}$.

\begin{proposition}[Tangent space I]\label{prop:tangent}
Let  $Q  \in \Gr(k,n)$ with eigendecomposition $Q = V I_{k,n-k} V^\tp$.
The tangent space of $\Gr(k,n)$ at $Q$ is given by
\begin{align}
\T_Q \Gr(k,n) &= \left\lbrace X\in \mathbb{R}^{n\times n}: X^\tp = X,\; X Q +QX = 0,\; \tr(X) = 0 \right\rbrace \label{eq:tan1}
 \\
& = \left\lbrace
V \begin{bmatrix}
0 & B \\
B^\tp & 0
\end{bmatrix}V^\tp  \in \mathbb{R}^{n \times n} : B \in \mathbb{R}^{k\times (n-k)}
\right\rbrace \label{eq:tan2} \\
& = \left\lbrace 
QV \begin{bmatrix}
0 & B \\
-B^\tp & 0 
\end{bmatrix} V^\tp  \in \mathbb{R}^{n \times n}: B\in \mathbb{R}^{k \times (n-k)} 
\right\rbrace. \label{eq:tan3}
\end{align}
\end{proposition}

\begin{proposition}[Riemannian metric]\label{prop:metric}
Let $Q\in \Gr(k,n)$ with $Q = V I_{k,n-k} V^\tp$ and
\[
X = 
V \begin{bmatrix}
0 & B \\
B^\tp & 0
\end{bmatrix}V^\tp
,\quad Y = 
V \begin{bmatrix}
0 & C \\
C^\tp & 0
\end{bmatrix}V^\tp
 \in \T_Q \Gr(k,n).
\]
Then
\begin{equation}\label{eq:metric}
\langle X, Y \rangle_Q \coloneqq \tr(XY) =2\tr(B^\tp C) 
\end{equation}
defines a Riemannian metric.
The corresponding Riemannian norm is
\begin{equation}\label{eq:norm}
\lVert X \rVert_Q \coloneqq \sqrt{\langle X, X \rangle}_Q= \lVert X \rVert_\F =\sqrt{2} \lVert B \rVert_\F.
\end{equation}
\end{proposition}

\begin{theorem}[Geodesics I]\label{thm:geo}
Let $Q\in \Gr(k,n)$ and $X \in \T_Q \Gr(k,n)$ with
\begin{equation}\label{eq:geotan}
Q = V I_{k,n-k} V^\tp, \qquad  X = V\begin{bmatrix}
0 & B \\
B^\tp & 0
\end{bmatrix} V^\tp.
\end{equation} 
The geodesic $\gamma$ emanating from $Q$ in the direction $X$ is given by
\begin{equation}\label{eq:geo}
\gamma(t) = V \exp \left( \frac{t}{2} \begin{bmatrix}
0 & -B \\
B^\tp & 0
\end{bmatrix} \right) I_{k,n-k} \exp \left( \frac{t}{2}\begin{bmatrix}
0 & B \\
-B^\tp & 0
\end{bmatrix} \right) V^\tp.
\end{equation}
The differential equation for $\gamma$ is
\begin{equation}\label{eq:geode}
\gamma(t)^\tp \ddot{\gamma}(t) -  \ddot{\gamma}(t)^\tp \gamma(t) = 0,\qquad \gamma(0) = Q,\qquad \dot{\gamma}(0) = X.
\end{equation}
\end{theorem}

\begin{proposition}[Parallel transport]\label{prop:pt}
Let $Q\in \Gr(k,n)$ and $X,Y \in \T_Q \Gr(k,n)$ with
\[
Q = V I_{k,n-k} V^\tp, \qquad
X = V\begin{bmatrix}
0 & B \\
B^\tp & 0
\end{bmatrix}V^\tp,
\qquad
Y = V\begin{bmatrix}
0 & C \\
C^\tp & 0
\end{bmatrix}V^\tp,
\]
where $V \in \O(n)$ and $B, C\in \mathbb{R}^{k \times (n-k)}$. 
Let $\gamma$ be a geodesic curve emanating from $Q$ in the direction $X$. Then the parallel transport of $Y$ along $\gamma$ is 
\begin{equation}\label{eq:pt}
Y (t) =V 
\exp \left( \frac{t}{2}\begin{bmatrix}
0 & -B \\
B^\tp & 0
\end{bmatrix}\right) \begin{bmatrix}
0 & C \\
C^\tp & 0
\end{bmatrix} \exp \left(\frac{t}{2} \begin{bmatrix}
0 & B \\
-B^\tp & 0
\end{bmatrix} \right) V^\tp.
\end{equation}
\end{proposition}
\subsection{some useful functions}
We recall the \emph{Peano--Baker series associated to a matrix function $\Phi: [a,b] \to \mathbb{R}^{n\times n}$}. To define the Peano--Baker series, we first recursively define a sequence $\{M_k(t)\}_{k=0}^{\infty}$ of matrix functions 
\begin{align*}
M_0(t) &= \I_n, \\
M_k(t) &= \I_n + \int_{a}^t \Phi(s) M_{k-1} (s) ds,\quad k\in \mathbb{N}.
\end{align*}
We have the following:
\begin{theorem}\cite[Section~3, Theorem~1]{Brockett70}\label{thm:PBseries}
The sequence $\{M_k(t)\}_{k=0}^{\infty}$ converges to a matrix function $M(t)$ uniformly on $[a,b]$, which solves the differential equation 
\[
\frac{d}{dt} X(t) = \Phi(t) X(t), \quad X(a) = \I_n. 
\]
In particular, given any column vector $u \in \mathbb{R}^n$, $M(t)u$ solves the differential equation 
\[
\frac{d}{dt} x(t) = \Phi(t ) x(t),\quad x(a) = u. 
\]
\end{theorem}
The limit matrix function $M(t)$ in Theorem~\ref{thm:PBseries} is defined to be the Peano--Baker series associated to $\Phi(t)$.
\subsection{vectorization of a matrix}
Let $m,n$ be positive integers and let $A[a_1,\dots, a_n]$ be a matrix of size $m\times n$ where $a_1,\dots, a_n \in \mathbb{R}^m$ are column vectors of $A$. We define the \emph{vectorization} of $A$ to be the column vector 
\[
\vect(A) \coloneqq \begin{bmatrix}
a_1 \\
\vdots \\
a_n
\end{bmatrix} \in \mathbb{R}^{mn}.
\]  
We recall that using vectorizations of matrices, we can express the matrix-matrix product in terms of matrix-vector product. Namely, for $A\in \mathbb{R}^{m\times n}$ and $B\in \mathbb{R}^{n\times l}$, we have 
\begin{equation}\label{eq:vectorization}
\vect (AB) = (\I_l \otimes A) \vect(B) = (B^\tp \otimes \I_m) \vect (A).
\end{equation}
Moreover, for any positive integers $m,n$, there exists a permutation matrix $K^{(m,n)} \in \mathbb{R}^{mn \times mn}$, called \emph{the commutation matrix} such that 
\begin{equation}\label{eq:commutation}
K^{(m,n)} \vect(A) = \vect(A^\tp),\quad A\in \mathbb{R}^{m\times n}.
\end{equation}
\section{Sub-Riemannian geometry of flag manifolds with classical embeddings}
According to \cite[Proposition~3.2]{YWL19}, $\Flag(n_1,\dots, n_d;n)$ can be naturally embedded into a product of Grassmann manifolds via
\begin{align}\label{eq:embedding1}
\iota: \Flag(n_1,\dots, n_d;n) &\hookrightarrow \Gr(n_1,n) \times \Gr(n_2-n_1,n) \cdots \times  \Gr(n_d- n_{d-1},n) \nonumber \\
( \left\lbrace \mathbb{V}_k\right\rbrace_{k=1}^d ) &\mapsto (\mathbb{W}_1,\mathbb{W}_2,\dots, \mathbb{W}_d).
\end{align}
Here $\mathbb{W}_1 = \mathbb{V}_1$ and $\mathbb{W}_k$ is the orthogonal complement of $\mathbb{V}_{k-1}$ in $\mathbb{V}_k, k=2,\dots,d$. For simplicity, we denote
\begin{equation}\label{eq:convention}
\boxed{m_1 \coloneqq n_1,\quad m_{d+1}  \coloneqq n - n_d, \quad m_{k} \coloneqq n_k - n_{k-1},\quad k = 2,\dots, d}
\end{equation}
so that $\iota$ is an embedding of $\Flag(n_1,\dots, n_d;n)$ into $\prod_{k=1}^d \Gr(m_k,n) $. 
\subsection{an embedding of a flag manifold into a matrix manifold}
By \eqref{eq:modelGr1}, we may also embed each $\Gr(m_k,n)$ into $\O(n) $ and hence we can write $\mathbb{W}_k$ in \eqref{eq:embedding1} as $V_k \I_{m_k, n - m_k} V_k^\tp$ for some $V_k\in \O(n)$. We denote by $\tau$ the induced embedding of $\prod_{k=1}^d \Gr(m_k,n)$ into $ \O(n)^d $. In the following, we will explicitly characterize the image $\tau \circ \iota\left( \Flag(n_1,\dots, n_d;n) \right)$ in $\O(n)^d$.
\begin{proposition}[embedding]\label{prop:modelFlag}
The image of the embedding
\begin{equation}\label{prop:modelFlag:eq0}
\varepsilon: \Flag(n_1,\dots, n_d;n) \xhookrightarrow{\iota} \prod_{k=1}^d \Gr(m_k,n) \xhookrightarrow{\tau}  \O(n)^d 
\end{equation}
is given by
\begin{multline}\label{prop:modelFlag:eq1}
\varepsilon \left(\Flag(n_1,\dots, n_d;n) \right) = 
\lbrace
(Q_1,\dots, Q_d)\in  \O(n)^d:\tr (Q_k) = 2m_k - n, Q_k^\tp = Q_k \\ (\I_n + Q_k)(\I_n + Q_{k+1}) = 0, k=1,\dots, d 
\rbrace.
\end{multline}
In particular, we have 
\begin{equation}\label{prop:modelFlag:eq2}
\varepsilon \left(\Flag(n_1,\dots, n_d;n) \right)  =\left\lbrace \left( V \J_1 V^\tp,\dots, V  \J_d V^\tp \right):V\in O(n) \right\rbrace,
\end{equation}
where $\J_k = \diag (-\I_{m_1},\cdots, -\I_{m_{k-1}}, \I_{m_k}, -\I_{m_{k+1}}, \cdots, -\I_{m_d},-\I_{m_{d+1}} )$ is obtained by permuting diagonal blocks of $\I_{m_k,n-m_k}$.
\end{proposition}
\begin{proof}
According to \eqref{eq:modelGr1}, we must have $\varepsilon (\{\mathbb{R}^n_k\}_{k=1}^d) = (Q_1,\dots, Q_d) \in \left( \O(n)\cap  \S_n \right)$ with $\rank Q_k = 2m_k - n$. Moreover, since $\mathbb{W}_k$ is perpendicular to $\mathbb{W}_{k+1}$, we must have $P_{\mathbb{W}_k} \circ P_{\mathbb{W}_{k+1}} = 0 $ where $P_{\mathbb{U}}$ is the orthogonal projection from $\mathbb{R}^n$ onto a subspace $\mathbb{U}$. Now by \cite[Proposition 2.3]{LLY20}, we have $P_{\mathbb{W}_k} = \frac{1}{2} (\I_n + Q_k)$ which proves \eqref{prop:modelFlag:eq1}. To see \eqref{prop:modelFlag:eq2}, we notice that the relation $(\I_n + Q_k)(\I_n + Q_{k+1}) = 0 $ implies that $Q_k Q_{k+1} = Q_{k+1} Q_k$ and hence there exists $V_0 \in \O(n)$ diagonalizing $Q_k$'s simultaneously, i.e., $Q_k = V_0 D_k V_0^\tp $ where $D_k$ is a diagonal matrix with $m_k$ $-1$'s and $(n-m_k)$ $1$'s along its diagonal. The restriction $(\I_n + Q_k)(\I_n + Q_{k+1}) = 0 $ forces $D_k =\sigma^\tp J_{k} \sigma$ for some permutation matrix $\sigma$ and hence $V \coloneqq \sigma V_0$ gives us the desired expression of $\varepsilon (\{\mathbb{R}^n_k\}_{k=1}^d)$ in \eqref{prop:modelFlag:eq2}.
\end{proof}
In fact, \eqref{prop:modelFlag:eq2} is a special case of the general fact \cite[page~384]{FH91} that $G/P$ is an adjoint orbit of $G$ if $P$ is a parabolic subgroup of a semi-simple Lie group $G$. In our case, we have $G = \O(n)$ and $P = \O(m_1) \times \cdots \times \O(m_{d+1})$ so that $G/P \simeq \Flag (n_1,\dots, n_d;n)$ is the adjoint orbit of $(J_1,\dots, J_d) \in O(n)^d$.

Due to Proposition~\ref{prop:modelFlag}, in the sequel we abuse the notation by also using $\Flag(n_1,\dots, n_d;n)$ to denote $\varepsilon \left(\Flag(n_1,\dots, n_d;n) \right)$. 
Accordingly, an element in $\Flag(n_1,\dots, n_d;n)$ is written as a $d$-tuple 
\[
(VJ_{1} V^\tp,\dots, VJ_{d} V^\tp) = V (J_{1},\dots, J_{d}) V^\tp
\] 
for some $V\in \O(n)$, where $m_1 = n_1$ and $m_k = n_k - n_{k-1}$ for $k = 2,\dots, d$.
\subsection{tangent space, Riemannian metric and normal space} We first consider the tangent space of $\Flag(n_1,\dots, n_d;n)$ at a point $ V (J_{1},\dots, J_{d}) V^\tp$. To do this, we take a curve $V(t)$ on $\O(n)$ such that $V(0) = V$. It is clear that $\Lambda \coloneqq V(0)^\tp \dot{V}(0)  \in \mathfrak{so}(n)$ and hence the tangent vector determined by the curve $V(t) (J_{1},\dots, J_{d}) V(t)^\tp$ is simply 
\[
\dot{V}(0)  (J_{1},\dots, J_{d}) V(0)^\tp  + V(0)  (J_{1},\dots, J_{d}) \dot{V}(0)^\tp
\]
which can be further written as 
\[
V(0)
\left(
 \Lambda (J_{1},\dots, J_{d})  -  (J_{1},\dots, J_{d})  \Lambda
 \right)
V(0)^\tp. 
\]
We partition $\Lambda$ as $\Lambda = (\Lambda(p,q))_{p,q=1}^{d+1}$ where $\Lambda(p,q)$ is a $m_p \times m_q$ matrix such that $\Lambda(q,p) =- \Lambda(p,q)^\tp$. This implies that
\begin{equation}\label{eq:basicalculation}
\Lambda J_{k}- J_{k} \Lambda  = -2 \begin{bmatrix}
0 & \cdots & 0   & \Lambda(k,1)^\tp & 0 &\cdots & 0\\
\vdots & \ddots & \vdots & \vdots & \vdots & \ddots & \vdots  \\
0 & \cdots & 0   & \Lambda(k,k-1)^\tp & 0 &\cdots & 0 \\
\Lambda(k,1) & \cdots & \Lambda(k,k-1) & 0 & \Lambda(k,k+1) & \cdots &\Lambda(k,d+1) \\
0 & \cdots & 0   & \Lambda(k,k+1)^\tp & 0 & \cdots & 0 \\ 
\vdots & \ddots & \vdots & \vdots & \vdots & \ddots & \vdots  \\
0 & \cdots & 0   & \Lambda(k,d+1)^\tp & 0 &\cdots & 0
\end{bmatrix}.
\end{equation}
We notice that there is a natural identification $\prod_{1\le j < k \le d+1} \mathbb{R}^{m_j \times m_k} \simeq \mathfrak{so}(n)$ and hence we have an injective map
\[
\psi: \prod_{1\le j < k \le d+1} \mathbb{R}^{m_j \times m_k} \simeq \mathfrak{so}(n) \hookrightarrow \prod_{j=1}^d \S_n,\quad \psi( (A_{jk})_{1\le j < k \le d+1} ) = \frac{1}{2} ( AJ  - JA),
\]
where $J =  (J_{1},\dots, J_{d})$ and $A \in \mathfrak{so}(n)$ is the skew-symmetric matrix uniquely determined by $(A_{jk})_{1\le j < k \le d+1}$. The above calculations can be summarized as the following
\begin{proposition}\label{prop:tangentFlag}
Given a point $\mathfrak{f} \coloneqq V (J_{1},\dots, J_{d}) V^\tp \in \Flag(n_1,\dots, n_d;n)$, the tangent space of   $\Flag(n_1,\dots, n_d;n)$ at $\mathfrak{f}$ is 
\[
\mathbb{T}_{\mathfrak{f}} \Flag(n_1,\dots, n_d;n) =  
V \left\lbrace
\psi( (A_{jk})_{1\le j < k \le d+1} ): A_{jk} \in \mathbb{R}^{m_j \times m_k}, 1\le j < k \le d+1
\right\rbrace V^\tp. 
\]
In other words, $\mathbb{T}_{\mathfrak{f}} \Flag(n_1,\dots, n_d;n)$ consists of vectors $V (X_1,\dots, X_d) V^\tp\in \prod_{j=1}^d S_n$ satisfying 
\begin{equation}\label{prop:tangentFlag:eq}
X_k(k,l) = -X_l(k,l), X_k (p,q) = 0, X_k(k,k) = 0, \quad 1\le k,l \le d, 1\le p,q \le d+1~\text{and}~p,q,l \ne k.
\end{equation}
Here for each $1\le s, t \le d+1$, $X_k(s,t)\in \mathbb{R}^{m_s \times m_t}$ denotes the $(s,t)$-th block of $X_k\in \S_n$ when we partition $X_k$ with respect to $n = m_1 + \cdots + m_d + m_{d+1}$.
\end{proposition}

Due to Proposition~\ref{prop:tangentFlag}, we are able to parametrize a curve on $\Flag(n_1,\dots, n_d;n)$ easily. 
\begin{corollary}[curves]\label{cor:curve}
If $c: (-\varepsilon,\varepsilon) \to \Flag(n_1,\dots, n_d;n)$ is a differentiable curve such that $c(0) =  V (J_{1},\dots, J_{d}) V^\tp$, then there exists a differentiable curve $\Lambda: (-\varepsilon,\varepsilon) \to  \mathfrak{so}(n)$ such that $\Lambda(k,k)(t) \equiv 0, k =1,\dots, d+1$ and 
\[
c(t) = V \exp(\Lambda(t)) (J_{1},\dots, J_{d}) \exp(-\Lambda(t)) V^\tp,
\] 
where $\Lambda(t) = (\Lambda(j,k))_{j,k=1}^{d+1,d+1}$ is the partition of $\Lambda(t)$ with respect to $n = m_1 + \cdots + m_{d+1}$. 
\end{corollary}

As a submanifold of $\prod_{k=1}^d \Gr(m_k,n)$ (or equivalently, $\prod_{k=1}^d \O(n) $), $\Flag(n_1,\dots, n_d;n)$ is equipped with an induced Riemannian metric: 
\begin{equation}\label{eq:metricFlag}
\langle V(X_1,\dots, X_d) V^\tp, V(Y_1,\dots, Y_d) V^\tp \rangle_{\mathfrak{f}} \coloneqq \sum_{k=1}^d \tr(X_k Y_k),
\end{equation}
where $\mathfrak{f} = V (J_{1},\dots, J_{d}) V^\tp$ is a point in $\Flag(n_1,\dots, n_d;n)$ and $V(X_1,\dots, X_d) V^\tp$, $V(Y_1,\dots, Y_d) V^\tp$ are tangent vectors of $\Flag(n_1,\dots, n_d;n)$ at $\mathfrak{f}$. More explicitly, we can write 
\begin{equation}\label{eq:metricFlag1}
 \langle V(X_1,\dots, X_d) V^\tp, V(Y_1,\dots, Y_d) V^\tp \rangle_{\mathfrak{f}} =2\sum_{k=1}^d \sum_{ l < k < m  }  \tr( X_k(l,k) Y_k(k,l) + X_k(m,k) Y_k(k,m)).
\end{equation}
We remark that summands in the formula \eqref{eq:metricFlag1} are not evenly counted. For example, if $d = 2$, then $ \langle V(X_1,X_2) V^\tp, V(Y_1,Y_2) V^\tp \rangle_{\mathfrak{f}} $ is
\begin{equation}\label{eq:metricFlag2}
2 (\tr (X_1(2,1)Y_1(1,2)) + \tr(X_1(3,1)Y_1(1,3)) + \tr(X_2(1,2)Y_2(2,1) ) + \tr(X_2(3,2) Y_2(2,3))),
\end{equation}
in which the coefficient of $\tr (X_1(2,1)Y_1(1,2)) $ is $4$ since $\tr(X_2(1,2)Y_2(2,1) ) = \tr (X_1(2,1)Y_1(1,2))$ while coefficients for the other two summands are both $2$.

For each $Q\in \O(n) $, we have $\mathbb{T}_Q \O(n)  = Q \mathfrak{so}(n)
$ and hence for each $(Q_1,\dots, Q_d) \in \prod_{k=1}^d \O(n)$, we obtain 
\[
\mathbb{T}_{(Q_1,\dots, Q_d) } \left( \prod_{k=1}^d \O(n) \right) = \bigoplus_{k=1}^d Q_k \mathfrak{so}(n).
\]
To calculate the normal space of $\Flag(n_1,\dots, n_d;n)$ in $\prod_{k=1}^d \O(n)$ at $\mathfrak{f} = V (J_{1},\dots, J_{d}) V^\tp$, we need to determine $Y_1,\dots, Y_d\in \mathfrak{so}(n)$ such that $Y \coloneqq ( V J_{1} V^\tp Y_1,\dots,  V J_{d} V^\tp Y_d)$ is perpendicular to $\mathbb{T}_{\mathfrak{f}} \Flag(n_1,\dots, n_d;n)$, i.e., $\langle X, Y \rangle_{\mathfrak{f},\O(n)} = 0$ for all $X \in \mathbb{T}_{\mathfrak{f}} \Flag(n_1,\dots, n_d;n)$. Here the inner product $\langle \cdot,\cdot \rangle_{\mathfrak{f}, \prod_{k=1}^ d  \O(n)}$ is the canonical Riemannian metric on $\prod_{k=1}^ d  \O(n)$ at the point $\mathfrak{f}$, which induces \eqref{eq:metricFlag}. We notice that for $X = V (X_1,\dots, X_d) V^\tp \in \mathbb{T}_{\mathfrak{f}} \Flag(n_1,\dots, n_d;n)$ 
\[
V X_k V^\tp = (V J_{k} V^\tp)  V J_{k} X_k V^\tp,\quad k =1,\dots, d,
\]
which implies that 
\begin{align*}
 \langle X,Y  \rangle_{\mathfrak{f},\prod_{k=1}^ d  \O(n)} &= \sum_{k=1}^d \tr( (V J_{k} X_k V^\tp)^\tp Y_k ) \\
 & = \sum_{k=1}^d \tr ( (V X_k J_{k} V^\tp ) Y_k ) \\
 & = \sum_{k=1}^d \tr ( (X_k J_{k}) V^\tp  Y_k V )
\end{align*}
Since $ \langle X,Y  \rangle_{\mathfrak{f},\prod_{k=1}^ d  \O(n)}  = 0$ holds for any $X \in \mathbb{T}_{\mathfrak{f}} \Flag(n_1,\dots, n_d;n)$, we can equivalently write this condition as 
\[
\sum_{k=1}^d \tr ( (X_k J_{k}) Z_k) = 0,\quad (X_1,\dots, X_d) \in \mathbb{T}_{\mathfrak{f}_0} \Flag(n_1,\dots, n_d;n),
\] 
where $\mathfrak{f}_0 = (J_{1},\dots, J_{d}) \in \Flag(n_1,\dots, n_d;n)$ and $Z_k = V^\tp Y_k V,k=1,\dots, d$. If we fix a pair $(k,l)$ such that $1\le k \le d, 1\le l \le d+1, k\ne l$ and set $X_m (p,q) = 0$ for
\[
(m,p,q) \not\in \left\lbrace  (k,k,l), (k,l,k), (l,k,l), (l,l,k) \right\rbrace,
\]
then since $X_k J_{m_k,n-m_k}$ is skew-symmetric, we have 
\footnotesize
\begin{align*}
0 =  \langle X,Y  \rangle_{\mathfrak{f},\prod_{k=1}^ d  \O(n)}  &= \tr (X_k(k,l) Z_k (l,k)) -  \tr(X_k (k,l)^\tp Z_k (k,l)) -  \tr( X_l (k,l)^\tp Z_l(k,l)) +  \tr( X_l(k,l) Z_l(l,k)) \\
&=  \tr (X_k(k,l) (Z_k (l,k)) - Z_l(l,k)) ) - \tr (X_k(k,l)^\tp (-Z_k (k,l)) + Z_l(k,l)) ) \\
& = \tr (X_k(k,l) (Z_k (l,k)) - Z_l(l,k)) ) + \tr (X_k(k,l)^\tp (Z_k (k,l)) - Z_l(k,l) ) \\
& = 2 \tr (X_k(k,l) ( Z_k(l,k) - Z_l (l,k) )).
\end{align*}\normalsize
Therefore, we may derive the following characterization of $\mathbb{N}_\mathfrak{f} \Flag(n_1,\dots, n_d;n)$:
\begin{proposition}\label{prop:normalspaceFlag}
At a point $\mathfrak{f}\coloneqq V (J_{1},\dots, J_{d}) V^\tp \in \Flag(n_1,\dots, n_d;\mathbb{R}^n)$, the normal space $\mathbb{N}_\mathfrak{f} \Flag(n_1,\dots, n_d;n)$ consists of vectors 
\[
(V J_{1} Z_1 V^\tp,\dots,  V J_{d} Z_d V^\tp)
\] 
where $Z_1,\dots,Z_d\in \mathfrak{so}(n)$ satisfy the relations 
\begin{itemize}
\item $Z_k(k,l) - Z_l (k,l)= 0$ for all $1\le k \ne l \le d$.
\item $Z_k(k,d+1) = 0,  Z_k(d+1,k) = 0$ for all $1 \le k \le d$.
\end{itemize}
 In particular, we have a decomposition 
\begin{equation}\label{prop:normalspaceFlag:eq1}
\mathbb{N}_\mathfrak{f} \Flag(n_1,\dots, n_d;n) = N_\mathfrak{f} \left( \prod_{k=1}^d \Gr(m_k,n) \right) \bigoplus N^0_\mathfrak{f}
\end{equation}
where $\mathbb{N}_\mathfrak{f} \left( \prod_{k=1}^d \Gr(m_k,n) \right) \coloneqq \prod_{k=1}^d \mathbb{N}_{V J_{m_k,n-m_k} V^\tp} \Gr(m_k,n)$ and 
\begin{multline}\label{prop:normalspaceFlag:eq2}
\mathbb{N}^0_\mathfrak{f} \coloneqq \lbrace 
(V J_{m_1,n-m_1} Z_1 V^\tp ,\dots,  V J_{m_d,n-m_d} Z_d V^\tp ): Z_k\in  \mathfrak{so}(n),Z_k(k,l)  - Z_l(k,l) = 0, \\
Z_k(k,k) = 0, Z_k(p,q) = 0, Z_k(k,d+1) = 0, Z_k(d+1,k) = 0, 1\le k,l \le d, 1\le p,q \le d+ 1, p,q\ne k
 \rbrace.
\end{multline}
\end{proposition}
We recall that $\Flag(n_1,\dots, n_d;n)$ can also be embedded into $\prod_{k=1}^d \Gr(n_k,n)$ as a Riemannian submanifold. Hence we may also characterize the normal space of $\Flag(n_1,\dots, n_d;n)$ with respect to this embedding.
\begin{corollary}
The normal space of $\Flag(n_1,\dots, n_d;n)$ in $\prod_{k=1}^d \Gr(m_k,n)$ at a point $\mathfrak{f} $ is $\mathbb{N}^0_\mathfrak{f} $.
\end{corollary}

\begin{proposition}[Projections]\label{prop:projection}
Projections from $\mathbb{T}_{\mathfrak{f}} \left(  \prod_{k=1}^d \O(n) \right)$ onto $\mathbb{T}_\mathfrak{f} \Flag(n_1,\dots, n_d;n)$ and $\mathbb{N}_\mathfrak{f} \Flag(n_1,\dots, n_d;n)$ are respectively given by 
\begin{align}\label{prop:projection:tangent}
\proj_{\mathfrak{f}}^{\mathbb{T}}: \mathbb{T}_{\mathfrak{f}} \left(  \prod_{k=1}^d \O(n) \right) &\to \mathbb{T}_\mathfrak{f} \Flag(n_1,\dots, n_d;n) \nonumber \\
V(J_{1} \Lambda_1 ,\dots, J_{d}\Lambda_d)  V^\tp &\mapsto V(X_1,\dots, X_d) V^\tp.
\end{align}
and 
\begin{align}\label{prop:projection:normal}
\proj_{\mathfrak{f}}^{\mathbb{N}}: \mathbb{T}_{\mathfrak{f}} \left(  \prod_{k=1}^d \O(n) \right) &\to \mathbb{N}_\mathfrak{f} \Flag(n_1,\dots, n_d;n) \nonumber \\
V(J_{1} \Lambda_1 ,\dots, J_{d}\Lambda_d)  V^\tp &\mapsto V( Z_1,\dots, Z_d) V^\tp
\end{align}
where for each $k = 1,\dots, d$, $X_k\in \S_n$ (resp. $Z_k\in \mathbb{R}^{n\times n}$) is partitioned as $(X_k (p,q))_{p,q =1}^{d+1}$ (resp. $(Z_k(p,q))_{p,q=1}^{d+1}$) with respect to $n = m_1 + \cdots + m_{d+1}$ and 
\[
X_k (p,q) = \begin{cases}
\frac{1}{2} (\Lambda_k(k,q) - \Lambda_{q} (k,q)),~\text{if}~p = k \ne q \le d \\ 
\Lambda_k (k,d+1),~\text{if}~p=k, q = d+1 \\
-\frac{1}{2} (\Lambda_k(p,k) - \Lambda_{p} (p,k)),~\text{if}~q = k \ne p \le d \\ 
-\Lambda_k (d+1,q),~\text{if}~q=k, p = d+1 \\
0,~\text{otherwise}.
\end{cases}
\]
\[
Z_k (p,q) = \begin{cases}
\frac{1}{2} (\Lambda_k(k,q) + \Lambda_{q} (k,q)),~\text{if}~p = k \ne q \le d \\ 
0,~\text{if}~p=k, q = d+1 \\
-\frac{1}{2} (\Lambda_k(p,k) + \Lambda_{p} (p,k)),~\text{if}~q = k \ne p \le d \\ 
0,~\text{if}~q=k, p = d+1 \\
\Lambda_k(p,q),~\text{otherwise}.
\end{cases}
\]
\end{proposition}

Before we proceed, we work out the case for $d = 2$ to exhibit our calculations above. In this case, our flag manifold is $\Flag(n_1,n_2;n)$ and hence $m_1 = n_1, m_2 = n_2 - n_1, m _3 = n - n_2$. A point $\mathfrak{f}$ in $\Flag(n_1,n_2;n)$ is written as 
\[
V (J_{1},J_{2}) V^\tp = 
V \left(  
\begin{bmatrix}
\I_{m_1} & 0 & 0 \\
0 & -\I_{m_2} & 0 \\
0 & 0 & - \I_{m_3}
\end{bmatrix}, 
\begin{bmatrix}
-\I_{m_1} & 0 & 0 \\
0 & \I_{m_2} & 0 \\
0 & 0 & - \I_{m_3}
\end{bmatrix}
\right) V^\tp,\quad V\in \O(n).
\]
A tangent vector of $\Flag(n_1,n_2;n)$ at $\mathfrak{f}$ is of the form 
\[
V \left(  
\begin{bmatrix}
0 & A & B \\
A^\tp & 0 & 0 \\
B^\tp & 0 & 0 
\end{bmatrix}, 
\begin{bmatrix}
0 & -A & 0 \\
-A^\tp & 0 & C \\
0 & C^\tp & 0 
\end{bmatrix}
\right) V^\tp,\quad A\in \mathbb{R}^{m_1 \times m_2}, B\in \mathbb{R}^{m_1 \times m_3}, C\in \mathbb{R}^{m_2 \times m_3}.
\]
The normal space of $\Flag(n_1,n_2;n)$ as a submanifold of $\O(n) \times \O(n)$ at $\mathfrak{f}$ consists of vectors 
\[
V \left(  
\begin{bmatrix}
X & Y & 0 \\
Y^\tp & Z & W \\
0 & -W^\tp & U 
\end{bmatrix}, 
\begin{bmatrix}
R & Y & S \\
Y^\tp & T & 0\\
-S^\tp & 0 & K 
\end{bmatrix}
\right) V^\tp, 
\]
where $X,R\in \mathfrak{so}(m_1)$, $Z,T\in \mathfrak{so}(m_2)$, $U,K\in \mathfrak{so}(m_3)$, $Y\in \mathbb{R}^{m_1 \times m_2},W\in \mathbb{R}^{m_2 \times m_3},S\in \mathbb{R}^{m_1 \times m_3}$. 

A tangent vector $\xi$ of $\O(n) \times \O(n)$ at $\mathfrak{f}$ can be written as
\[
\xi \coloneqq V \left(  
\begin{bmatrix}
A & B & C \\
B^\tp & D & E \\
C^\tp & -E^\tp & F 
\end{bmatrix}, 
\begin{bmatrix}
X & Y & Z \\
Y^\tp & W & S\\
-Z^\tp & S^\tp & T 
\end{bmatrix}
\right) V^\tp,
\]
where $A,X\in \mathfrak{so}(m_1),D,W\in \mathfrak{so}(m_2), F,T\in \mathfrak{so}(m_3), B,Y\in \mathbb{R}^{m_1 \times m_2}, C,Z\in \mathbb{R}^{m_1 \times m_3}, E,S\in \mathbb{R}^{m_2 \times m_3}$. The projection of $\xi$ onto $T_\mathfrak{f} \Flag(n_1,n_2;n)$ is 
\[
\proj^{\mathbb{T}}_{\mathfrak{f}} (\xi) = V \left(  
\begin{bmatrix}
0 & \frac{B-Y}{2} & C \\
\frac{B^\tp-Y^\tp}{2} & 0 & 0 \\
C^\tp & 0 & 0 
\end{bmatrix}, 
\begin{bmatrix}
0 & -\frac{B-Y}{2} & 0 \\
-\frac{B^\tp - Y^\tp}{2} & 0 & S\\
0 & S^\tp & 0 
\end{bmatrix}
\right) V^\tp
\]
and its projection onto $N_{\mathfrak{f}} \Flag(n_1,n_2;n)$ is 
\[
\proj^{\mathbb{N}}_{\mathfrak{f}} (\xi) = V \left(  
\begin{bmatrix}
A & \frac{B + Y}{2} & 0 \\
\frac{B^\tp + Y^\tp}{2} & D & E \\
0 & -E^\tp & F 
\end{bmatrix}, 
\begin{bmatrix}
X & \frac{B + Y}{2} & Z \\
\frac{B^\tp + Y^\tp}{2} & W & 0\\
-Z^\tp & 0 & T 
\end{bmatrix}
\right) V^\tp
\]
The normal space $\mathbb{N}_f^0$ of $\Flag(n_1,n_2;n)$ as a submanifold of $\Gr(m_1,n) \times \Gr(m_2,n)$ at $\mathfrak{f}$ consists of vectors 
\[
V \left(  
\begin{bmatrix}
0 & Y & 0 \\
Y^\tp & 0 & 0 \\
0 & 0 & 0 
\end{bmatrix}, 
\begin{bmatrix}
0 & Y & 0 \\
Y^\tp & 0 & 0\\
0 & 0 & 0
\end{bmatrix}
\right) V^\tp, \quad Y\in \mathbb{R}^{m_1 \times m_2}.
\]
We also recall that the tangent space $\mathbb{T}_{\mathfrak{f}} (\Gr(m_1,n) \times \Gr(m_2,n))$ consists of vectors 
\[
V \left(  
\begin{bmatrix}
0 & A & B \\
A^\tp & 0 & 0 \\
B^\tp & 0 & 0 
\end{bmatrix}, 
\begin{bmatrix}
0 & D & 0 \\
D^\tp & 0 & C \\
0 & C^\tp & 0 
\end{bmatrix}
\right) V^\tp,\quad V\in \O(n), A,D\in \mathbb{R}^{m_1 \times m_2}, B\in \mathbb{R}^{m_1 \times m_3}, C\in \mathbb{R}^{m_2 \times m_3}.
\]
The following identities can be directly verified by the above computations. 
\begin{align*}
\mathbb{T}_{\mathfrak{f}} \left( \O(n) \times \O(n) \right)  &=\mathbb{T}_{\mathfrak{f}} \Flag(n_1,n_2;n) \bigoplus \mathbb{N}_{\mathfrak{f}}  \Flag(n_1,n_2;n),  \\
\mathbb{T}_{\mathfrak{f}} (\Gr(m_1,n) \times \Gr(m_2,n)) &=\mathbb{T}_{\mathfrak{f}} \Flag(n_1,n_2;n)  \bigoplus \mathbb{N}_{\mathfrak{f}}^0.
\end{align*}
\subsection{geodesics}\label{subsec:geodesic}
Recall that we may parametrize a curve $c(t)$ on $\Flag(n_1,\dots, n_d;n)$ as 
\[
c(t) = V(t) (J_{1},\dots, J_{d})  V^\tp (t),
\] 
where $V(t)$ is a curve in $\O(n)$. By differentiating the equation $V(t)^\tp V(t) = \I_n$, we obtain 
\[
\dot{V}(t)^\tp V(t) + V(t)^\tp \dot{V}(t) = 0,
\]
from which we may write $\dot{V}(t)$ as 
\[
\dot{V}(t) =V(t) \Lambda(t),
\]
for some $\Lambda(t)\in \mathfrak{so}(n)$. According to Proposition~\ref{prop:tangentFlag}, we may further partition $\Lambda(t)$ as 
\[
\Lambda(t) = (\Lambda_{jk})_{j,k=1}^{d+1,d+1}
\]
with respect to $n = m_1 + \cdots + m_{d+1}$ and $\Lambda_{kk}(t) \equiv 0, k =1,\dots, d+1$. Hence the second derivative of $c(t)$ is 
\[
\ddot{c}(t) =  V(t) \left( \Delta_1(t), \dots, \Delta_d(t) \right) V(t)^\tp
\]
where 
\begin{equation}\label{eq:2ndderivative}
\Delta_k (t) = (\dot{\Lambda}(t) J_k - J_k \dot{\Lambda(t)}) + ({\Lambda}^2(t) J_k + J_k {\Lambda}^2(t)) + \left( - 2 {\Lambda}(t) J_k {\Lambda}(t) \right),\quad k =1,\dots, d.
\end{equation}
We may rewrite $\ddot{c}(t)$ as 
\[
\ddot{c}(t) = T_1(t) + T_2(t) -2 T_3(t)
\]
where $T_j(t)$ is the $j$-summand of $V(t) \left( \Delta_1(t), \dots, \Delta_d(t) \right) V(t)^\tp$ with respect to the decomposition of $\Delta_k(t)$ given in \eqref{eq:2ndderivative}. More precisely, 
\begin{align}
T_1(t) &= V(t) (\dot{\Lambda}(t) J_1 - J_1 \dot{\Lambda}(t),\dots, \dot{\Lambda}(t) J_d - J_d \dot{\Lambda}(t)) V(t)^\tp, \\
T_2(t) &= V(t) ({\Lambda}^2(t) J_1 + J_1 {\Lambda}^2(t),\dots, {\Lambda}^2(t) J_d + J_d {\Lambda}^2(t)) V(t)^\tp,\\
T_3(t) &= V(t) ( {\Lambda}(t) J_1 {\Lambda}(t),\dots,  {\Lambda}(t) J_d {\Lambda}(t)) V(t)^\tp.
\end{align}

We recall that the geodesic equation on $\Flag(n_1,\dots, n_d;n)$ is given by 
\[
\proj_{c(t)}^{\mathbb{T}} (\ddot{c}(t)) = 0.
\] 
Therefore, to determine the geodesic equation explicitly, we need to compute the projections of $T_1(t),T_2(t),T_3(t)$ to $T_{c(t)} \Flag(n_1,\dots, n_d;n)$ respectively. From Proposition~\ref{prop:tangentFlag}, $T_1(t)$ already lies in the tangent space $T_{c(t)} \Flag(n_1,\dots, n_d;n)$. Hence it is sufficient to determine the projections of $T_2(t)$ and $T_3(t)$. 

\begin{lemma}\label{lemma:projectionT2}
Let $c(t), T_2(t)$ be as above. The projection of $\proj_{c(t)}^{\mathbb{T}} (T_2(t))$ is zero.
\end{lemma}
\begin{proof}
We first compute ${\Lambda}^2(t) J_k + J_k {\Lambda}^2(t)$ for each $k =1,\dots, d$. To do this, we partition ${\Lambda}^2(t)$ (resp. $J_k$) as $(\Gamma_{p,q}(t))$ (resp. $(J_k(p,q))$) with respect to the partition $n = m_1 + \cdots + m_{d+1}$ and we recall that 
\[
J_k (p,q) = \begin{cases}
(2\delta_{pk} - 1 )\I_{m_p},\quad~\text{if}~q = p, \\
0,\quad~\text{otherwise}.
\end{cases}
\]
Here $\delta_{pk}$ is the Kronecker delta function. Since ${\Lambda}(t)$ is skew-symmetric, ${\Lambda}^2(t)$ is symmetric. We have $\Gamma_{q,p} = \Gamma_{p,q}^\tp$. Now the $(p,q)$-th block of ${\Lambda}^2(t) J_k$ is  
\[
\sum_{l=1}^{m+1} \Gamma_{p,l} J_k (l,q) = \Gamma_{p,q} J_k(q,q) = (2\delta_{qk} - 1) \Gamma_{p,q}
\]
and the $(p,q)$-th block of $J_k {\Lambda}^2(t) = ({\Lambda}^2(t) J_k)^\tp$ is $(2\delta_{pk} - 1) \Gamma_{p,q}$. This implies that the $(p,q)$-th block of ${\Lambda}^2(t) J_k + J_k {\Lambda}^2(t)$ is 
\[
(2\delta_{qk} - 1) \Gamma_{p,q} +  (2\delta_{pk} - 1) \Gamma_{p,q} = (-2) (1 - \delta_{pk} - \delta_{qk}) \Gamma_{p,q}.
\]
In particular, if either $q \ne p = k$ or $p\ne q = k$, we obtain that the $(p,q)$-th block of ${\Lambda}^2(t) J_k + J_k {\Lambda}^2(t)$ is zero and this implies that $\proj_{c(t)}^{\mathbb{T}} (T_2(t)) = 0$.
\end{proof}

\begin{lemma}\label{lemma:projectionT3}
Let $c(t),T_3(t)$ be as before. The projection $\proj_{c(t)}^{\mathbb{T}} (T_3(t))$ is 
\[
V(t) (X_1,\dots, X_d) V(t)^\tp
\]
where for each $1\le k \le d$, $X_k$ is a symmetric matrix whose $(p,q)$-th block vanishes for any $(p,q)$ except $(k,d+1)$ and $(d+1,k)$. Moreover if we partition $\Lambda(t)$ as $\Lambda(t) = \begin{bmatrix}
\Lambda_0(t) & \Lambda_1(t) \\
-\Lambda_1(t)^\tp & 0
\end{bmatrix} $ where $\Lambda_0(t) \in \mathfrak{so}(n-m_{d+1})$ and $\Lambda_1(t) \in \mathbb{R}^{(n-m_{d+1}) \times m_1}$ we have 
\[
\begin{bmatrix}
X_1(1,d+1) \\
\vdots \\
X_d(d,d+1)
\end{bmatrix} = -{\Lambda}_0(t) {\Lambda}_1(t).
\]
\end{lemma}
\begin{proof}
It is sufficient to compute $X_k \coloneqq {\Lambda}(t) J_k {\Lambda(t)}$ for each $k = 1,\dots, d$. We again partition $\Lambda(t)$ as $(\Lambda(p,q)(t))_{p,q=1}^d$ with respect to $n = m_1 + \cdots + m_{d+1}$. The $(p,q)$-th block of ${\Lambda}(t) J_k {\Lambda(t)}$ is 
\begin{align}
\sum_{l,s=1}^{d+1}{\Lambda}(p,l)(t) J_k (l,s) {\Lambda}(s,q)(t) &= \sum_{l=1}^{d+1} {\Lambda}(p,l)(t) J_k (l,l) {\Lambda}(l,q)(t) \nonumber \\
&=\sum_{l=1}^{d+1} (2\delta_{kl} - 1) {\Lambda}(p,l)(t) {\Lambda}(l,q)(t). \label{lemma:projectionT3:eq1}
\end{align}
In particular, for $1\le q\ne k \le d$, the $(k,q)$-th block of ${\Lambda}(t) J_k {\Lambda(t)}$ is 
\[
\sum_{l=1}^{d+1} (2\delta_{kl} - 1) {\Lambda}(k,l)(t) {\Lambda}(l,q)(t),
\]
while the $(k,q)$-th block of ${\Lambda}(t) J_q {\Lambda(t)}$ is 
\[
\sum_{l=1}^{d+1} (2\delta_{ql} - 1) {\Lambda}(k,l)(t) {\Lambda}(l,q)(t).
\]
Using Proposition~\ref{prop:projection}, we may conclude that the $(k,q)$-th block of $X_k$ is zero if $1 \le k, q \le d$.

If we take $q = d +1$ and $p = k$ in \eqref{lemma:projectionT3:eq1}, then the $(k,d+1)$-th block of $X_k$ is 
\[
X_k(k,d+1) =- \sum_{1\le l \ne k \le d}  {\Lambda}(k,l)(t) {\Lambda}(l,d+1)(t).
\]
We observe that $X_k(k,d+1)$ is the $k$-th block of the product
\[
-\begin{bmatrix}
0 & {\Lambda(1,2)}(t) & \dots & {\Lambda}
(1,d-1)(t) & {\Lambda}
(1,d)(t) \\
{\Lambda}(2,1)(t) & 0 & \dots & {\Lambda}(2,d-1)(t) & {\Lambda}(2,d)(t)  \\
\vdots &\vdots & \ddots & \vdots & \vdots \\
{\Lambda}(d-1,1)(t) & {\Lambda}(d-1,2)(t)  & \dots &  0 & {\Lambda}(d-1,d)(t) \\ 
{\Lambda}(d,1)(t) & {\Lambda}(d,2)(t)  & \dots & {\Lambda}(d,d-1)(t) &0
\end{bmatrix}
\begin{bmatrix}
{\Lambda}(1,d+1)(t) \\
{\Lambda}(2,d+1)(t) \\
\vdots \\
{\Lambda}(d-1,d+1)(t) \\
{\Lambda}(d,d+1)(t) \\
\end{bmatrix},
\]
which can be written in a compact form $-{\Lambda}_0(t) {\Lambda}_1(t)$.
\end{proof}

By assembling Lemmas~\ref{lemma:projectionT2} and \ref{lemma:projectionT3}, we can easily derive the geodesic equation on a flag manifold, from which we can even obtain an explicit formula for the geodesic curve. In fact, we have the following:
\begin{proposition}[geodesics]\label{prop:gedoesic}
Let $c(t)$ be a curve on $\Flag(n_1,\dots, n_d;n)$. We parametrize $c(t)$ as 
\[
c(t) = V(t)  (J_{1},\dots, J_{d})  V(t)^\tp,
\] 
where $V(t)$ is a curve in $\O(n)$. We have the following: 
\begin{enumerate}
\item There exists a unique $\Lambda(t) \in \mathfrak{so}(n)$ such that $\dot{V}(t) = V(t) \Lambda(t)$. \label{prop:gedoesic:item1}
\item If we partition $\Lambda(t)$ as $\Lambda(t) = (\Lambda(p,q)(t))_{p,q=1}^{d+1,d+1}\in \mathfrak{so}(n)$ with respect to $n = m_1 + \cdots + m_{d+1}$, then $\Lambda(p,p)(t) \equiv 0, p =1,\dots, d+1$.\label{prop:gedoesic:item2}
\item $c(t)$ is a geodesic curve if and only if 
\begin{equation}\label{prop:gedoesic:eq1}
\dot{\Lambda}_0(t) = 0,\quad \dot{\Lambda}_1(t) = {\Lambda}_0(t) {\Lambda}_1(t).
\end{equation}
where $\Lambda_0(t) \coloneqq (\Lambda(p,q)(t))_{p,q=1}^{d,d}$ and $\Lambda_1(t) \coloneqq (\Lambda(d+1,q)(t))_{q=1}^{d}$. \label{prop:gedoesic:item3}
\item The solution to \eqref{prop:gedoesic:eq1} is 
\[
\Lambda_0 (t) = {\Lambda}_0(0),\quad \Lambda_1(t) =\exp(  {t \Lambda}_0(0)) \Lambda_1(0) .
\]
Hence a geodesic curve $c(t)$ is
\[
c(t) = V(t)  (J_1,\dots, J_d)   V^\tp(t),
\]
where $V(t)$ is a curve in $\O(n)$ written as 
\begin{equation}\label{prop:gedoesic:eq2}
V(t) = V(0) \exp \left( t\begin{bmatrix}
2X_0 &  X_1 \\
-X_1^\tp & 0 
\end{bmatrix}  \right) \begin{bmatrix}
\exp(-tX_0) & 0 \\
0 & \I_{m_{d+1}} \
\end{bmatrix}
\end{equation}
for some $X_0 \in \mathfrak{so}(n-m_{d+1})$ satisfying $X_0(k,k) = 0, k =1,\dots, d$ and $X_1 \in \mathbb{R}^{(n-m_{d + 1}) \times m_{d+1}}$.\label{prop:gedoesic:item4}
\end{enumerate}
\end{proposition}
\begin{proof}
\eqref{prop:gedoesic:item1}--\eqref{prop:gedoesic:item3} and the first half of \eqref{prop:gedoesic:item4} are obvious from our earlier discussions, hence it is only left to prove the second part of \eqref{prop:gedoesic:item4}. To that end, we notice that $V(t)$ must satisfy the equation 
\begin{equation}\label{prop:gedoesic:eq3}
\dot{V}(t) = V(t) \begin{bmatrix}
X_0 & \exp (t X_0)X_1 \\
 -X_1^\tp \exp(-t X_0) & 0 
\end{bmatrix}
\end{equation}
and 
\[
\begin{bmatrix}
X_0 & \exp (t X_0)X_1 \\
 -X_1^\tp \exp(-t X_0) & 0 
\end{bmatrix} =  \begin{bmatrix}
\exp(tX_0) & 0 \\
0 & \I_{m_{d+1}} \
\end{bmatrix} \begin{bmatrix}
X_0 & X_1 \\
-X_1^\tp & 0
\end{bmatrix} 
\begin{bmatrix}
\exp(-tX_0) & 0 \\
0 & \I_{m_{d+1}} \
\end{bmatrix}.
\]
If we set $W(t) = V(t) \begin{bmatrix}
\exp(tX_0) & 0 \\
0 & \I_{m_{d+1}} \
\end{bmatrix}$, then \eqref{prop:gedoesic:eq3} becomes 
\[
\dot{W}(t) = W(t) \begin{bmatrix}
2X_0 &  X_1 \\
-X_1^\tp & 0 
\end{bmatrix}
\]
whose solution is simply $W(t) = W(0) \exp \left( t \begin{bmatrix}
2X_0 &  X_1 \\
-X_1^\tp & 0 
\end{bmatrix}  \right) = V(0) \exp \left( t \begin{bmatrix}
2X_0 &  X_1 \\
-X_1^\tp & 0 
\end{bmatrix}  \right) $. Hence we obtain that 
\[
V(t) = V(0) \exp \left( t \begin{bmatrix}
2X_0 &  X_1 \\
-X_1^\tp & 0 
\end{bmatrix}  \right) \begin{bmatrix}
\exp(-tX_0) & 0 \\
0 & \I_{m_{d+1}} \
\end{bmatrix}.
\]
\end{proof}

We remark that if $d = 1$, then $X_0 = 0$ in \eqref{prop:gedoesic:eq2} and a geodesic curve on $\Gr(n_1,n)$ passing through $V J_1 V^\tp$ is 
\[
c(t) = V \exp \left( t \begin{bmatrix}
0 & X_1 \\
-  X_1^\tp & 0
\end{bmatrix} \right) \I_{n_1,n-n_1}   \left( -t \begin{bmatrix}
0 & X_1 \\
-  X_1^\tp & 0
\end{bmatrix} \right) V^\tp,
\]
which coincides with the formula derived in \cite{LLY20}.

We again work out the case $d = 2$ to illustrate the proof of Proposition~\ref{prop:gedoesic}. To this end, we write 
\[
\Lambda (t) = \begin{bmatrix}
0 & A(t) & B(t) \\
-A^\tp (t) & 0 & C(t) \\
-B^\tp (t) & -C^\tp(t) & 0 
\end{bmatrix},\quad A(t)\in \mathbb{R}^{m_1 \times m_2},B(t)\in \mathbb{R}^{m_1 \times m_3},C(t)\in \mathbb{R}^{m_2 \times m_3}
\]
and suppose that the curve 
\[
c(t) = V(t) (J_1,J_2) V(t)^\tp,\quad \dot{V}(t) = V(t) \Lambda(t), \quad V(t)\in \O(n)
\]
is a curve passing through $(J_1,J_2)$ with the direction
\[
({\Lambda}(0) J_1 - J_1 {\Lambda}(0), {\Lambda}(0)J_2- J_2 {\Lambda}(0)) =-2 \left( \begin{bmatrix}
0 & A(0) & B(0) \\
A(0)^\tp & 0 & 0 \\
B^\tp(0) & 0 & 0 
\end{bmatrix}, 
\begin{bmatrix}
0 & -A(0) & 0 \\
-A^\tp(0) & 0 & C(0) \\
0 & C^\tp(0) & 0 
\end{bmatrix} \right).
\]
We write  $\ddot{c}(t) = V(t)\left( \Delta_1(t), \Delta_2(t) \right) V(t)^\tp$ where 
\[
\Delta_k (t) = (\dot{\Lambda}(t) J_k - J_k \dot{\Lambda(t)}) + ({\Lambda}^2(t) J_k + J_k {\Lambda}^2(t)) + \left( - 2 {\Lambda}(t) J_k {\Lambda}(t) \right).
\]
It is sufficient to compute the projection of ${\Lambda}(t) J_k {\Lambda}(t)$ onto $T_{c(t)} \Flag(n_1,n_2;n)$, which is 
\[
{\Lambda}(t) J_1 {\Lambda}(t) =  
\left( \begin{bmatrix}
*  & {B}(t){C}(t)^\tp & -{A}(t){C}(t) \\ 
{C}(t){B}(t)^\tp & * & * \\
-{C}(t)^\tp {A}(t)^\tp & * & *
\end{bmatrix}, \begin{bmatrix}
*  & {B}(t){C}(t)^\tp & * \\ 
{C}(t){B}(t)^\tp & * & {A}(t)^\tp {B}(t) \\
* & {B}(t)^\tp {A}(t) & *
\end{bmatrix} 
\right),
\]
where $*$ denotes those irrelevant blocks. Eventually, we obtain 
\scriptsize{
\[
\proj^{\mathbb{T}}_{c(t)} (\dot{c}(t)) =-2 \left( \begin{bmatrix}
0  &  \dot{A}(t) & \dot{B}(t) -{A}(t){C}(t) \\ 
\dot{A}(t)^\tp & 0 & 0 \\
\dot{B}(t)^\tp- {C}(t)^\tp {A}(t)^\tp & 0 & 0
\end{bmatrix}, \begin{bmatrix}
0 & -\dot{A}(t) & 0 \\ 
-\dot{A}^\tp(t) & 0 & \dot{C}(t) + {A}(t)^\tp {B}(t) \\
0 & \dot{C}(t)^\tp + {B}(t)^\tp {A}(t) & 0
\end{bmatrix} 
\right).
\]\normalsize
Hence the geodesic equation for $\Flag(n_1,n_2;n)$ is 
\[
\dot{A}(t) =0, \quad \dot{B}(t) - {A}(t){C}(t)= 0, \quad  \dot{C}(t) + {A}(t)^\tp {B}(t)= 0,
\]
which can be rewritten in a more compact form: 
\begin{equation}\label{eq:geodesic:d=2}
\dot{A} = 0,\quad \begin{bmatrix}
\dot{B}(t) \\
\dot{C}(t)
\end{bmatrix} = \begin{bmatrix}
0 & {A}(t) \\
-{A}^\tp(t) & 0
\end{bmatrix} 
\begin{bmatrix}
{B}(t) \\
{C}(t)
\end{bmatrix}.
\end{equation}
The solution to \eqref{eq:geodesic:d=2} is 
\[
A(t) = {A}(0),\quad \begin{bmatrix}
B(t) \\
C(t)
\end{bmatrix} = \exp\left(t \begin{bmatrix}
0 & {A}(0) \\
-{A}^\tp(0) & 0
\end{bmatrix} \right)  \begin{bmatrix}
{B}(0) \\
{C}(0)
\end{bmatrix}.
\]

\section{Sub-Riemannian geometry of flag manifolds with modified embeddings} In this section, we discuss the embedded geometry of flag manifolds with respect to a modified version of the embedding \eqref{eq:embedding1}. Namely, we define 
\begin{align}\label{eq:embedding2}
\tilde{\iota}: \Flag(n_1,\dots, n_d;n) &\hookrightarrow \Gr(n_1,n) \times \Gr(n_2-n_1,n) \cdots \times  \Gr(n_d- n_{d-1},n) \times \Gr(n-n_d,n) \nonumber \\
( \left\lbrace \mathbb{V}_k\right\rbrace_{k=1}^d ) &\mapsto (\mathbb{W}_1,\mathbb{W}_2,\dots, \mathbb{W}_d,\mathbb{W}_{d+1}),
\end{align}
Here $\mathbb{W}_k$ is the orthogonal complement of $\mathbb{V}_{k-1}$ in $\mathbb{V}_k$ for $2\le k \le d$, $\mathbb{W}_1 = \mathbb{R}^n_1$ and $\mathbb{W}_{d+1}$ is the orthogonal complement of $\mathbb{V}_d$ in $\mathbb{R}^n$. We observe that 
\[
\tilde{\iota}( \left\lbrace \mathbb{V}_k\right\rbrace_{k=1}^d ) = (\iota( \left\lbrace \mathbb{V}_k\right\rbrace_{k=1}^d ), \mathbb{W}_{d+1}).
\]
In other words, $\tilde{\iota}$ is simply an extension of $\iota$ by tautologically adding the orthogonal complement of $\mathbb{V}_d$. Since $\iota$ is already an embedding, we may easily conclude that $\tilde{\iota}$ is also an embedding. Adopting the convention \eqref{eq:convention}, $\tilde{\iota}$ embeds $\Flag(n_1,\dots, n_d;n)$ into $\prod_{j=1}^{d+1} \Gr(m_j,n)$. Moreover, by Proposition~\ref{prop:modelFlag} we have the following:
\begin{proposition}[embedding]\label{prop:newmodelFlag}
The image of the embedding
\begin{equation}\label{prop:newmodelFlag:eq0}
\tilde{\varepsilon}: \Flag(n_1,\dots, n_d;n) \xhookrightarrow{\tilde{\iota}} \prod_{j=1}^{d+1} \Gr(m_j,n) \xhookrightarrow{\tilde{\tau}}  \O(n)^{d+1} 
\end{equation}
is given by
\begin{multline}\label{prop:newmodelFlag:eq1}
\tilde{\varepsilon} \left(\Flag(n_1,\dots, n_d;n) \right) = 
\lbrace
(Q_1,\dots, Q_{d+1})\in\prod_{j=1}^{d+1}  \O(n):\tr (Q_j) = 2m_j - n, Q_j^\tp = Q_j \\ (\I_n + Q_j)(\I_n + Q_{j+1}) = 0, j=1,\dots, d+1 
\rbrace.
\end{multline}
In particular, we also have 
\begin{equation}\label{prop:newmodelFlag:eq2}
\tilde{\varepsilon} \left(\Flag(n_1,\dots, n_d;n) \right)  =\left\lbrace  V (J_1,\dots, J_{d+1}) V^\tp: V\in O(n) \right\rbrace,
\end{equation}
where $\J_k = \diag (-\I_{m_1},\cdots, -\I_{m_{k-1}}, \I_{m_k}, -\I_{m_{k+1}}, \cdots, -\I_{m_{d+1}})$ is obtained by permuting diagonal blocks of $\I_{m_k,n-m_k}, k =1,\dots, d+1$ and 
\[
V (J_1,\dots, J_{d+1}) V^\tp \coloneqq \left( V \J_1 V^\tp,\dots, V  \J_{d+1} V^\tp \right).
\]
\end{proposition}
Similarly to Proposition~\ref{prop:tangentFlag} and Corollary~\ref{cor:curve}, we also have: 
\begin{proposition}\label{prop:newmodeltangentFlag}
Given a point $\tilde{\mathfrak{f}} \coloneqq V (J_1,\dots, J_{d+1}) V^\tp$, the tangent space $\mathbb{T}_{\tilde{\mathfrak{f}}} \Flag(n_1,\dots, n_d;n)$ consists of vectors $V (X_1,\dots, X_{d+1}) V^\tp \in \prod_{j=1}^{d+1} S_n$ satisfying
\begin{equation}\label{prop:newmodeltangentFlag:eq}
X_k(k,l) = - X_{l}(k,l), X_k(p,q) = 0, X_k(k,k) = 0, \quad 1\le k,l,p,q \le d+1~\text{and}~p,q,l \ne k.
\end{equation} 
Here $X_k(s,t) \in \mathbb{R}^{m_s \times m_t}$ is the $(s,t)$-th block of $X_k\in S_n$ when we partition $X_k$ with respect to $n = \sum_{j=1}^{d+1} m_j$. Moreover, a curve $c(t)$ passing through $c(0) = V (J_1,\dots, J_{d+1}) V^\tp$ on $\Flag(n_1,\dots, n_d;n)$ can be locally parametrized as 
\[
c(t) = V \exp(\Lambda(t)) (J_1,\dots, J_{d+1}) \exp(-\Lambda(t)) V^\tp.
\]
For some differentiable curve $\Lambda: (-\varepsilon,\varepsilon) \to \mathfrak{so}(n)$ such that $\Lambda(k,k)(t) \equiv 0$.
\end{proposition}
If $d = 2$, then a tangent vector of $\Flag(n_1, n_2;n)$ at $\tilde{f}$ can be written as 
\[
V \left(  
\begin{bmatrix}
0 & A & B \\
A^\tp & 0 & 0 \\
B^\tp & 0 & 0 
\end{bmatrix}, 
\begin{bmatrix}
0 & -A & 0 \\
-A^\tp & 0 & C \\
0 & C^\tp & 0 
\end{bmatrix}, 
\begin{bmatrix}
0 & 0 & -B \\
0 & 0 & -C \\
-B^\tp & -C^\tp & 0 
\end{bmatrix}
\right) V^\tp,
\]
where $A\in \mathbb{R}^{m_1 \times m_2}, B\in \mathbb{R}^{m_1 \times m_3}, C\in \mathbb{R}^{m_2 \times m_3}$.
\subsection{induced Riemannian metric, normal space and projections}
As a submanifold of $\prod_{j =1}^{d+1} \O(n)$, $\Flag(n_1,\dots, n_d;n)$ is equipped with a naturally induced Riemannian metric:
\begin{align}\label{eq:newmodelmetricFlag}
\langle V(X_1,\dots, X_{d+1}) V^\tp, V(Y_1,\dots, Y_{d+1}) V^\tp \rangle_{\tilde{\mathfrak{f}}} &\coloneqq \sum_{j=1}^{d+1} \tr(X_j Y_j) \\
\nonumber
&= 2\sum_{k=1}^{d+1} \sum_{ l < k < m}  \tr( X_k(l,k) Y_k(k,l) + X_k(m,k) Y_k(k,m)).
\end{align}
Unlike \eqref{eq:metricFlag} in which some summands are weighted differently, all summands in the new metric \eqref{eq:newmodelmetricFlag} are evenly weighted. For instance, if we take $d =2$ then $\langle V(X_1,X_2, X_{3}) V^\tp, V(Y_1,Y_2,Y_{3}) V^\tp \rangle_{\tilde{\mathfrak{f}}} $ is simply 
\begin{equation}\label{eq:newmodelmetricFlag1}
4 \left( \tr (X_1(2,1) Y_1(1,2)) +  \tr (X_1(3,1) Y_1(1,3)) + \tr (X_2(3,2) Y_2(2,3)) \right).
\end{equation}
The distinction between \eqref{eq:newmodelmetricFlag} and \eqref{eq:metricFlag} can be easily observed by comparing  \eqref{eq:newmodelmetricFlag1} with \eqref{eq:metricFlag2}.

We notice that the tangent space of $\prod_{j =1}^{d+1} \O(n)$ at $\tilde{\mathfrak{f}} = V (J_1,\dots, J_{d+1}) V^\tp$ is 
\[
\mathbb{T}_{\tilde{\mathfrak{f}}} \left( \prod_{j =1}^{d+1} \O(n) \right)= \bigoplus_{j=1}^{d+1} \left( V J_j V^\tp \mathfrak{so}(n) \right) = \bigoplus_{j=1}^{d+1} \left( V J_j  \mathfrak{so}(n) V^\tp \right).
\]
\begin{proposition}\label{prop:newmodelnormalspaceFlag}
At a point $\tilde{\mathfrak{f}}\coloneqq V (J_{1},\dots, J_{d+1}) V^\tp \in \Flag(n_1,\dots, n_d;\mathbb{R}^n)$, the normal space $\mathbb{N}_{\tilde{\mathfrak{f}}} \Flag(n_1,\dots, n_d;n)$ consists of vectors 
\[
V ( J_{1} Z_1,\dots,   J_{d+1} Z_{d+1} )V^\tp
\] 
where $Z_1,\dots,Z_d\in \mathfrak{so}(n)$ satisfy the relation
\[
Z_k(k,l) - Z_l (k,l)= 0,\quad \text{for all}~1\le k \ne l \le d+1.
\]
In particular, we have a decomposition 
\begin{equation}\label{prop:newmodelnormalspaceFlag:eq1}
\mathbb{N}_{\tilde{\mathfrak{f}}} \Flag(n_1,\dots, n_d;n) = N_{\tilde{\mathfrak{f}}} \left( \prod_{k=1}^{d+1} \Gr(m_k,n) \right) \bigoplus N^0_{\tilde{\mathfrak{f}}},
\end{equation}
where $\mathbb{N}_{\tilde{\mathfrak{f}}} \left( \prod_{k=1}^{d+1} \Gr(m_k,n) \right) = \bigoplus_{k=1}^{d+1} \mathbb{N}_{V J_{m_k,n-m_k} V^\tp} \Gr(m_k,n)$ and 
\begin{multline}\label{prop:newmodelnormalspaceFlag:eq2}
\mathbb{N}^0_{\tilde{\mathfrak{f}}} = \lbrace 
V( J_{m_1,n-m_1} Z_1  ,\dots,   J_{m_{d+1},n-m_{d+1}} Z_{d+1}  )V^\tp: Z_k\in  \mathfrak{so}(n),Z_k(k,l)  - Z_l(k,l) = 0, \\
Z_k(k,k) = 0, Z_k(p,q) = 0, 1\le k,l,p,q \le d+1, p,q\ne k
 \rbrace.
\end{multline}
\end{proposition}

\begin{proposition}[Projections]\label{prop:newmodelprojection}
Projections from $\mathbb{T}_{\tilde{\mathfrak{f}}} \left(  \prod_{k=1}^{d+1} \O(n) \right)$ onto $\mathbb{T}_{\tilde{\mathfrak{f}}} \Flag(n_1,\dots, n_d;n)$ and $\mathbb{N}_{\tilde{\mathfrak{f}}} \Flag(n_1,\dots, n_d;n)$ are respectively given by 
\begin{align}\label{prop:newmodelprojection:tangent}
\proj_{\tilde{\mathfrak{f}}}^{\mathbb{T}}: \mathbb{T}_{\tilde{\mathfrak{f}}} \left(  \prod_{k=1}^{d+1} \O(n) \right) &\to \mathbb{T}_{\tilde{\mathfrak{f}}} \Flag(n_1,\dots, n_d;n) \nonumber \\
V(J_{1} \Lambda_1 ,\dots, J_{d+1}\Lambda_{d+1})  V^\tp &\mapsto V(X_1,\dots, X_{d+1}) V^\tp,
\end{align}
and 
\begin{align}\label{prop:newmodelprojection:normal}
\proj_{\tilde{\mathfrak{f}}}^{\mathbb{N}}: \mathbb{T}_{\tilde{\mathfrak{f}}} \left(  \prod_{k=1}^{d+1} \O(n) \right) &\to \mathbb{N}_{\tilde{\mathfrak{f}}} \Flag(n_1,\dots, n_d;n) \nonumber \\
V(J_{1} \Lambda_1 ,\dots, J_{d+1}\Lambda_{d+1})  V^\tp &\mapsto V( Z_1,\dots, Z_{d+1}) V^\tp,
\end{align}
where for each $k = 1,\dots, d$, $X_k\in \S_n$ (resp. $Z_k\in \mathbb{R}^{n\times n}$) is partitioned as $(X_k (p,q))_{p,q =1}^{d+1}$ (resp. $(Z_k(p,q))_{p,q=1}^{d+1}$) with respect to $n = m_1 + \cdots + m_{d+1}$ and 
\[
X_k (p,q) = \begin{cases}
\frac{1}{2} (\Lambda_k(k,q) - \Lambda_{q} (k,q)),~\text{if}~p = k \ne q  \\ 
-\frac{1}{2} (\Lambda_k(p,k) - \Lambda_{p} (p,k)),~\text{if}~q = k \ne p  \\ 
0,~\text{otherwise}.
\end{cases}
\]
\[
Z_k (p,q) = \begin{cases}
\frac{1}{2} (\Lambda_k(k,q) + \Lambda_{q} (k,q)),~\text{if}~p = k \ne q  \\ 
-\frac{1}{2} (\Lambda_k(p,k) + \Lambda_{p} (p,k)),~\text{if}~q = k \ne p  \\ 
\Lambda_k(p,q),~\text{otherwise}.
\end{cases}
\]
\end{proposition}
As an illustrative example, we take a tangent vector $\xi$ of $\O(n) \times \O(n) \times \O(n)$ at some point $\mathfrak{f} = V (J_1,J_2,J_3) V^\tp$, which can be written as
\[
\xi \coloneqq V \left(  
\begin{bmatrix}
A & B & C \\
B^\tp & D & E \\
C^\tp & -E^\tp & F 
\end{bmatrix}, 
\begin{bmatrix}
X & Y & Z \\
Y^\tp & W & S\\
-Z^\tp & S^\tp & T 
\end{bmatrix},
\begin{bmatrix}
L & M & N \\
-M^\tp & P & Q\\
N^\tp & Q^\tp & R 
\end{bmatrix}
\right) V^\tp,
\]
where $A,X,L\in \mathfrak{so}(m_1)$, $D,W,P\in \mathfrak{so}(m_2)$, $F,T,R\in \mathfrak{so}(m_3)$, $B,Y,M\in \mathbb{R}^{m_1 \times m_2}$, $C,Z,N\in \mathbb{R}^{m_1 \times m_3}$, $E,S,Q\in \mathbb{R}^{m_2 \times m_3}$. The projection of $\xi$ to $T_{\tilde{\mathfrak{f}}} \Flag(n_1,n_2;n)$ is 
\small
\[
\proj^{\mathbb{T}}_{\tilde{\mathfrak{f}}} (\xi) = V \left(  
\begin{bmatrix}
0 & \frac{B-Y}{2} & \frac{C-N}{2} \\
\frac{B^\tp-Y^\tp}{2} & 0 & 0 \\
\frac{C^\tp-N^\tp}{2} & 0 & 0 
\end{bmatrix}, 
\begin{bmatrix}
0 & -\frac{B-Y}{2} & 0 \\
-\frac{B^\tp - Y^\tp}{2} & 0 & \frac{S-Q}{2}\\
0 & \frac{S^\tp - Q^\tp}{2} & 0 
\end{bmatrix},
\begin{bmatrix}
0 & 0 & -\frac{C-N}{2} \\
0 & 0 & -\frac{S-Q}{2}\\
-\frac{C^\tp-N^\tp}{2} & -\frac{S^\tp - Q^\tp}{2} & 0 
\end{bmatrix}
\right) V^\tp
\]\normalsize
and its projection to $N_{\mathfrak{f}} \Flag(n_1,n_2;n)$ is 
\[
\proj^{\mathbb{N}}_{\mathfrak{f}} (\xi) = V \left(  
\begin{bmatrix}
A & \frac{B + Y}{2} & \frac{C+N}{2} \\
\frac{B^\tp + Y^\tp}{2} & D & E \\
\frac{C^\tp+N^\tp}{2} & -E^\tp & F 
\end{bmatrix}, 
\begin{bmatrix}
X & \frac{B + Y}{2} & Z \\
\frac{B^\tp + Y^\tp}{2} & W & \frac{S+Q}{2}\\
-Z^\tp & \frac{S^\tp+Q^\tp}{2} & T 
\end{bmatrix},
\begin{bmatrix}
L & M & \frac{C+N}{2} \\
-M^\tp & P & \frac{S+Q}{2}\\
\frac{C^\tp+N^\tp}{2} & \frac{S^\tp + Q^\tp}{2} & T 
\end{bmatrix}
\right) V^\tp
\]\normalsize
\subsection{geodesics}
Assume that $c(t)$ is a curve in $\Flag(n_1,\dots,n_d;n)$, then according to Proposition~\ref{prop:newmodeltangentFlag} we may parametrize $c(t)$ as 
\begin{equation}\label{eq:newmodelcurve}
c(t) = V(t)  (J_1,\dots, J_{d+1}) V(t)^\tp
\end{equation}
for some differentiable curve $V(t)$ in $\O(n)$. Moreover, we have $\dot{V}(t) = V(t) \Lambda(t)$ where $\Lambda(t)$ is a curve in $\mathfrak{so}(n)$ partitioned as $\Lambda(t) = (\Lambda(p,q))_{p,q=1}^{d+1}$ with respect to $m_1 + \cdots + m_{d+1} = n$ and and 
$\Lambda(k,k)(t) \equiv 0, k =1,\dots, d+1$. This implies that we have 
\[
\ddot{c}(t) = T_1(t) + T_2(t) -2 T_3(t),
\]
where $T_j(t)$'s are respectively given by 
\begin{align}
T_1(t) &= V(t) (\dot{\Lambda}(t) J_1 - J_1 \dot{\Lambda(t)},\dots, \dot{\Lambda}(t) J_{d+1} - J_{d+1} \dot{\Lambda(t)}) V^\tp(t), \\
T_2(t) &= V(t) ({\Lambda}^2(t) J_1 + J_1 {\Lambda}^2(t),\dots, {\Lambda}^2(t) J_{d+1} + J_{d+1} {\Lambda}^2(t)) V^\tp(t),\\
T_3(t) &= V(t) ( {\Lambda}(t) J_1 {\Lambda}(t),\dots,  {\Lambda}(t) J_{d+1} {\Lambda}(t)) V^\tp(t).
\end{align}
By similar calculations in proofs of Lemmas~\ref{lemma:projectionT2} and \ref{lemma:projectionT3}, we may easily obtain the following characterizations of $\proj_{c(t)}^{\mathbb{T}} (T_j(t)), j = 1,2,3$.
\begin{lemma}\label{lemma:newmodelprojectionTj}
Let $c(t), \Lambda(t), T_1(t), T_2(t), T_3(t)$ be as above. We have 
\begin{enumerate}
\item $T_1(t) \in \mathbb{T}_{c(t)} \Flag(n_1,\dots, n_d;n)$.
\item $\proj_{c(t)}^{\mathbb{T}} (T_2(t)) = 0$.
\item $\proj_{c(t)}^{\mathbb{T}} (T_3(t)) = 0$.
\end{enumerate}
\end{lemma}
\begin{proposition}\label{prop:newmodelgedoesic}
Let $c(t)$ be a curve on $\Flag(n_1,\dots,n_d;n)$ parametrized as 
\[
c(t) = V(t)  (J_1,\dots, J_{d+1})  V(t)^\tp
\]
for some differentiable curve $V(t)$ in $\O(n)$. Let $\Lambda(t)$ be the curve in $\mathfrak{so}(n)$ such that $\dot{V}(t) = V(t) \Lambda(t)$, where $\Lambda(t)$ is a curve in $\mathfrak{so}(n)$ partitioned as $\Lambda(t) = (\Lambda(p,q))_{p,q=1}^{d+1}$ with respect to $m_1 + \cdots + m_{d+1} = n$ and and 
$\Lambda(k,k)(t) \equiv 0, k =1,\dots, d+1$. Then $c(t)$ is a geodesic curve if and only if $V(t) =V(0) \exp(t \Lambda(0))$.
\end{proposition}
\begin{proof}
Since $c(t)$ is a geodesic if and only if $\proj_{c(t)} (\ddot{c}(t)) \equiv 0$, Lemma~\ref{lemma:newmodelprojectionTj} implies that $c(t)$ is a geodesic curve if and only if 
\[
\dot{\Lambda}(t) J_k - J_k \dot{\Lambda(t)} = 0,\quad k =1,\dots, d+1.
\]
By \eqref{eq:basicalculation}, we may conclude that $c(t)$ is a geodesic if and only if $\dot{\Lambda}(t) \equiv 0$, i.e., $\Lambda(t) = \Lambda(0)$. This implies that $V(t)$ is determined by the equation $\dot{V}(t) = V(t) \Lambda(0)$, from which we may conclude that $V(t) =V(0) \exp(t \Lambda(0))$.
\end{proof}

\section{The comparison of Riemannian metrics on flag manifolds} The goal of this section is to discuss relations among three Riemannian metrics on a flag manifold $\Flag(n_1,\dots, n_d;n)$. We recall that the two metrics discussed in this paper are respectively induced by the embedding $\varepsilon: \Flag(n_1,\dots, n_d;n) \hookrightarrow \prod_{k=1}^{d} \O(n)$ given in \eqref{prop:modelFlag:eq0} and $\tilde{\varepsilon}: \Flag(n_1,\dots, n_d;n) \hookrightarrow \prod_{k=1}^{d+1} \O(n)$ given in \eqref{prop:newmodelFlag:eq0}. For notational simplicity, we denote the two induced metrics by $g^e$ and $\tilde{g}^e$, respectively. Yet there is another metric induced from the homogeneous space structure of $\Flag(n_1,\dots, n_d;n)$, which is discussed thoroughly in \cite{YWL19}. We denote this quotient metric by $g^q$. 

\begin{proposition}
The Riemaannian metrics $\tilde{g}^e$ and $g^q$ coincide. Moreover, $\tilde{g}^e$ and $g^e$ coincide with $g^q$ when $d = 1$, in which case $\Flag(n_1,\dots, n_d;n)$ is simply the Grassmann manifold $\Gr(n_1;n)$.
\end{proposition}

We will see in Proposition~\ref{prop:construction} that both $g^e$ and $\tilde{g}^e = g^q$ can be constructed by a uniform method. To begin with, we notice that in general, any smooth map 
\[
\varphi:  \left( \mathbb{R}^{n\times n} \right)^d \to  \mathbb{R}^{n\times n}
\]
induces an embedding $\kappa_{\varphi}: \left( \mathbb{R}^{n\times n} \right)^d \to \left( \mathbb{R}^{n\times n} \right)^{d+1}$ defined by 
\[
\kappa_{\varphi}(A_1,\dots, A_d) = (A_1,\dots, A_d, \varphi(A_1,\dots,A_d)),\quad A_j\in \mathbb{R}^{n\times n}, j=1,\dots, d.
\]
Hence we have another embedding $\kappa_{\varphi} \circ \varepsilon$ of $\Flag(n_1,\dots, n_d;n)$ into $\O(n)^{d+1} \subseteq  \left( \mathbb{R}^{n\times n}\right)^{d+1}$, which induces a metric $g^{\varphi}$ on $\Flag(n_1,\dots, n_d;n)$ from the Euclidean metric on $\left( \mathbb{R}^{n\times n} \right)^{d+1}$.

\begin{proposition}\label{prop:construction}
We have the following:
\begin{itemize}
\item $g^{\varphi} = g^e$ if and only if $\varphi$ is a constant map on $\varepsilon(\Flag(n_1,\dots, n_d;n))$. In particular, $g^{\varphi} = g^e$ if $\varphi$ is a constant map.
\item There exists $\varphi$ such that $g^{\varphi} = \widetilde{g}^e$.
\end{itemize}
\end{proposition}
\begin{proof}
The ``if" part of the first statement can be verified by a straightforward calculation. For the ``only if" part, we notice that $g^{\varphi} = g^e$ implies that the differential map $d_{(Q_1,\dots, Q_d)}\varphi$ must be zero on $T_{(Q_1,\dots, Q_d)} \varepsilon(\Flag(n_1,\dots, n_d;n))$ at any $(Q_1,\dots, Q_d) \in \varepsilon(\Flag(n_1,\dots, n_d;n))$. Since $\varepsilon(\Flag(n_1,\dots, n_d;n))$ is connected and $\varphi$ is continuous, we may conclude that $\varphi$ is a constant map on $\varepsilon(\Flag(n_1,\dots, n_d;n))$. 

For the second statement, we notice that $C \coloneqq \varepsilon(\Flag(n_1,\dots, n_d;n))$ is a compact subset of $X \coloneqq \left( \mathbb{R}^{n\times n} \right)^d $ and we can define 
\[
\psi: C \to \O(n) \subseteq \mathbb{R}^{n\times n},\quad \psi(Q_1,\dots, Q_d) = Q_{d+1},
\]
where $(Q_{d+1} + \I_n)/2$ is the projection matrix of $\left( \bigoplus_{j=1}^d \im(Q_j + \I_n) \right)^\perp$. We denote by $p_{ij}$ the projection map from $\mathbb{R}^{n\times n}$ onto its $(i,j)$-th entry, $1\le i,j\le n$. It is clear that $p_{ij}\circ \psi: C \to \mathbb{R}$ is a smooth function. The compactness of $C$ in $X$ implies that $p_{ij}\circ \psi$ has a smooth extension $\varphi_{ij}: X \to \mathbb{R}$. Indeed, we can first extend the function $p_{ij}\circ \psi$ smoothly to an open neighbourhood of $C$ and then further extend it smoothly to the whole $X$ by a smooth partition of unity. Now we have a smooth map 
\[
\varphi \coloneqq (\varphi_{ij}): \left( \mathbb{R}^{n\times n} \right)^d \to \mathbb{R}^{n\times n}
\]
which extends $\psi$ and hence we have $g^{\varphi} = \widetilde{g}^e$.
\end{proof}
\section{A coordinate descent method for optimizations on flag manifolds}
Given a strictly increasing sequence $n_1 < \cdots < n_d$, we define
\[
m_1 \coloneqq n_1,\quad m_{d+1}  \coloneqq n - n_d, \quad m_{j} \coloneqq n_j - n_{j-1},\quad j = 2,\dots, d+1.
\]
We recall from \eqref{eq:embedding2} that a flag $\{\mathbb{V}_k\}_{k=1}^d \in \Flag(n_1,\dots, n_d;n)$ can be regarded as $\{\mathbb{W}_j\}_{j=1}^{d+1}$ via the modified embedding $\widetilde{\iota}:\Flag(n_1,\dots, n_d;n) \hookrightarrow \prod_{j=1}^{d+1} \Gr(m_j,n)$, where $\mathbb{W}_j$ is the orthogonal complement of $\mathbb{V}_{j-1}$ in $\mathbb{V}_{j}, 2 \le j \le d+1$, $\mathbb{W}_1 = \mathbb{V}_1$ and $\mathbb{V}_{d+1} = \mathbb{R}^n$. Therefore, an optimization problem on $\Flag(n_1,\dots, n_d;n)$ has the following form:
\begin{align}\label{eq:optimization on flag}
\min \quad & f( \mathbb{W}_1,\dots, \mathbb{W}_{d+1}) \nonumber \\
\text{s.t.} \quad & \mathbb{W}_j \in \Gr(m_j, n), 1\le j \le d + 1 \\
\quad & \mathbb{W}_j \perp \mathbb{W}_l, 1 \le j < l \le d + 1  \nonumber
\end{align}
Here $f$ is a function on $\Flag(n_1,\dots, n_d;n)$. We propose Algorithm~\ref{alternating}, an alternating type algorithm to solve the problem~\eqref{eq:optimization on flag}.
\begin{algorithm}[!htbp] \label{alg:cord}
\caption{Coordinate minimization method for optimization on flag manifolds}\label{alternating}
\begin{algorithmic}[1]
\renewcommand{\algorithmicrequire}{\textbf{Input}}
\Require
A differentiable function $f$ on $\Flag(n_1,\dots, n_d;n)$
\renewcommand{\algorithmicrequire}{\textbf{Output}}
\Require
A critical point of $f$
\renewcommand{\algorithmicrequire}{\textbf{Initialization}}
\Require
Choose an initial point $(\mathbb{W}_{1},\dots, \mathbb{W}_{d+1}) \in \prod_{j=1}^{d+1} \Gr(m_j,n)$ 
\While{not converge}
\State {set $(s,t) = (1,2)$}
\For{$1\le s < t \le d+1$}
\State{Solve the following sub-problem for $(\mathbb{X}_s,\mathbb{X}_t) \in \Gr(m_s,n)\times \Gr(m_t, n)$:
\begin{align}\label{eq:subproblem}
\min \quad & f( \mathbb{W}_1,\dots,\mathbb{W}_{s-1},\mathbb{X}_s,\mathbb{W}_{s+1},\dots, \mathbb{W}_{t-1},\mathbb{X}_t,\mathbb{W}_{t+1},\dots \mathbb{W}_{d+1}) \nonumber \\
\text{s.t.} \quad & \mathbb{X}_s \perp \mathbb{X}_t \\
\quad & \mathbb{X}_s \perp \mathbb{W}_j, 1 \le j \ne s \le d + 1 \nonumber \\
\quad & \mathbb{X}_t \perp \mathbb{W}_j, 1 \le j \ne t \le d + 1 \nonumber
\end{align}}
\State Update $(\mathbb{W}_s,\mathbb{W}_t)$ by the solution $(\overline{\mathbb{X}}_s, \overline{\mathbb{X}}_t)$ to \eqref{eq:subproblem}.
\State Update $(s,t)$ by $(s+1,t)$ if $s+1 < t$ and by $(s,t+1)$ otherwise
\EndFor
\EndWhile
\end{algorithmic}
\end{algorithm}

We remark that the sub-problem \eqref{eq:subproblem} in Algorithm~\ref{alternating} is an optimization problem on a Grassmann manifold. Indeed, we notice that $\mathbb{W}_{j}$ in \eqref{eq:subproblem} is fixed whenever $j\ne s,t$. This implies
\[
\mathbb{X}_s \oplus \mathbb{X}_t  = \left(\bigoplus_{j\ne s,t} \mathbb{W}_j\right)^\perp
\]
is a fixed $(m_s + m_t)$-dimensional subspace of $\mathbb{R}^n$. So the submanifold given by fixed $\mathbb{W}_{j}, j\neq s,t$ and $\mathbb{X}_s \oplus \mathbb{X}_t$ is isomorphic to $\Gr(m_s, m_s + m_t)$. This submanifold is actually a totally-geodesic manifold, which is clear from the geodesic formulas of flag and Grassmann manifolds. Thus the objective function 
\[
f( \mathbb{W}_1,\dots,\mathbb{W}_{s-1},\mathbb{X}_s,\mathbb{W}_{s+1},\dots, \mathbb{W}_{t-1},\mathbb{X}_t,\mathbb{W}_{t+1},\dots \mathbb{W}_{d+1})
\]
can be recognized as a function on the submanifold $\Gr(m_s,m_s + m_t)$. Furthermore, at a given point, there are $d(d+1)/2$ such submanifolds indexed by $1 \leq s < t \leq d+1$. The tangent spaces of those submanifolds are orthogonal to each other and span the whole tangent space. Algorithm~\ref{alternating} is a generalization of coordinate descent algorithm in Euclidean space.

\subsection{separation of subspaces}
Given $d+1$ subspaces $\mathbb{U}_1,\dots,\mathbb{U}_{d+1}$ of some ambient space $\mathbb{R}^N$. The separation problem can be mathematically formulated as the following optimization problem on a flag manifold:
\begin{align}\label{eq:nearest point on flag}
\min \quad & F(\mathbb{W}) := \sum_{j=1}^{d+1} \lVert \tau_j(\mathbb{U}_j) - \tau_j(\mathbb{W}_j) \rVert_F^2 \nonumber \\
\text{s.t.} \quad & \mathbb{W}_j \in \Gr(m_j, n), 1\le j \le d + 1 \\
\quad & \mathbb{W}_j \perp \mathbb{W}_l, 1 \le j < l \le d + 1  \nonumber
\end{align}
Here $m_j = \dim \mathbb{U}_j, 1\le j \le d+1$, $n = \sum_{j=1}^{d+1} m_j$ and $\tau_j$ is the embedding of $\Gr(m_j,n)$ into $\O(n) \cap \S_n$ defined by 
\[
\tau_j (\mathbb{W}) = V \begin{bmatrix}
-\I_{p} & 0 & 0 \\
0 & \I_{m_j} & 0 \\
0 & 0 & -\I_{q}
\end{bmatrix} V^\tp
\]
where $p = \sum_{l=1}^{j-1} m_l$, $q = \sum_{l=j+1}^{d+1} m_l$ and $V = [v_1,\dots, v_{n}] = [V_1, \dots, V_{d+1}]\in \O(n)$ such that $[v_{p+1},\dots, v_{q-1}] = V_j, \operatorname{span}\{v_{p+1},\dots, v_{q-1}\} = \mathbb{W}_j$. 

\begin{lemma}\label{lem:gradient estimate}
Consider the maximization of linear function $f(Q) = \langle A, Q \rangle$ on the Grassmann manifold,
\begin{align*}
\max \quad & \langle A, Q \rangle  \\
\text{s.t.} \quad & Q \in \Gr(k, n) 
\end{align*}
The gradient of $f(Q)$ is given by 
\[
\nabla f(Q) = \frac{1}{4}(A+A^\tp - QAQ - QA^\tp Q).
\]
Let $(A+A^\tp)/2 = U \Lambda U^\tp$ be an eigendecomposition of $(A+A^\tp)/2$ such that $\Lambda = \diag(\lambda_1, \dots, \lambda_n)$, $\lambda_1 \geq \dots \geq \lambda_n$. Then $Q^* = UI_{k, n-k} U^\tp$ is a maximizer of $f(Q)$. Furthermore, 
\[
2\|\Lambda\| (f(Q^*) - f(Q)) \geq \|\nabla f(Q)\|^2.
\]
\end{lemma}
\begin{proof}
The formula for gradient is given in \cite[Proposition~5.1]{LLY20}. The original problem is equivalent to
\begin{align*}
\max \quad & \langle \Lambda, Q \rangle  \\
\text{s.t.} \quad & Q \in \Gr(k, n)
\end{align*}
and we need to prove $Q = I_{k, n-k}$ is a maximizer. Using gradient formula, we can simplify the first order condition $\nabla f(Q^*) = 0$ to
\[
Q^* \Lambda = \Lambda Q^*.
\]
So $Q^*, A$ can be simultaneously diagonalized, and we can assume $Q^*$ is diagonalized. The original problem is equivalent to
\[
\min_{\delta_1 + \dots + \delta_n =2k -n, \delta_i = \pm 1} \lambda_1 \delta_1+ \dots + \lambda_n \delta_n.
\]
It is clear that $\delta_1 = \dots = \delta_k = 1, \delta_{k+1} = \dots = \delta_n = -1$ is a maximizer. So $Q^* = I_{k, n-k}$ is a maximizer. Now consider the last inequality. The term $\|\nabla f(Q)\|^2, f(Q^*) - f(Q)$ can be simplified
\begin{align*}
\|\nabla f(Q)\|^2 &= \frac{1}{4}\langle \Lambda - Q \Lambda Q, \Lambda - Q \Lambda Q \rangle\\
&= \frac{1}{2}\sum_{i = 1}^n \lambda_i^2 - \frac{1}{2}\tr(\Lambda Q \Lambda Q)\\
&= \frac{1}{2}\tr(\Lambda Q^* \Lambda Q^*) - \frac{1}{2}\tr(\Lambda Q \Lambda Q),
\end{align*}
\begin{align*}
f(Q^*) - f(Q) &= \langle \Lambda, Q^* \rangle - \langle \Lambda, Q \rangle.
\end{align*}
For any $c > 2\|\Lambda\|$, we have
\begin{align*}
c(f(Q^*) - f(Q)) - \|\nabla f(Q)\|^2 &= c \langle \Lambda, Q^* \rangle - c \langle \Lambda, Q \rangle - \tr(\Lambda Q^* \Lambda Q^*)/2 +\tr(\Lambda Q \Lambda Q)/2\\
&= g(Q^*) - g(Q),
\end{align*}
where $g(Q) = c \langle \Lambda, Q \rangle - \tr(\Lambda Q \Lambda Q)/2$. Assume $Q^{**}$ is a maximizer of $g(Q)$. The first order condition of $g(Q)$ is
\[
Q^{**}\Lambda Q^{**} \Lambda Q^{**} - c Q^{**} \Lambda Q^{**} - \Lambda Q^{**} \Lambda + c \Lambda = 0,
\]
which is equivalent to
\[
(Q^{**}\Lambda - \Lambda Q^{**})(Q^{**} \Lambda  + \Lambda Q^{**} - c I) = 0.
\]
By definition $c > 2\|\Lambda\| \geq \|Q^{**} \Lambda  + \Lambda Q^{**}\|$, so $Q^{**} \Lambda  + \Lambda Q^{**} - c I$ is invertible and $Q^{**}\Lambda = \Lambda Q^{**}$. So $Q^{**}, \Lambda$ can be simultaneously diagonalized. We can assume $Q^{**}$ is diagonalized. So
\[
g(Q^{**}) = \sum_{i=1}^n (2\|\Lambda\| \lambda_i \delta_i - \frac{1}{2} \lambda_i^2),
\]
where $\delta_i$ is the diagonal of $Q^{**}$. Again, $Q^*$ is a maximizer of $g(Q)$. So we have proved that $g(Q^*) \geq g(Q)$, i.e.,
\[
c(f(Q^*) - f(Q)) - \|\nabla f(Q)\|^2 \geq 0.
\]
Because $c$ is any number larger than $2\|\Lambda\|$, it also holds for $c = 2\|\Lambda\|$ and the proof is finished.
\end{proof} 

\begin{proposition}
If we apply Algorithm~\ref{alternating} to solve the problem~\eqref{eq:nearest point on flag}, then for each $1 \le s < t \le d+1$, the sub-problem has the form
\begin{align}\label{eq:alternating nearest point on flag}
\min \quad & \lVert A_1 - W \I_{m_s, m_t} W^\tp\rVert_F^2 + \lVert 
A_2 + W \I_{m_s, m_t} W^\tp \rVert_F^2  \\
\text{s.t.} \quad & W \in \O(m_s+m_t)  \nonumber
\end{align}
where $A_1,A_2 \in \O(m_s+m_t) \cap S_{m_s+m_t}$ are some fixed matrices. Moreover, the sub-problem has an explicit solution $W_\ast$ which is given by the SVD of $A_1 - A_2 = W_\ast \Sigma W_\ast^\tp$.

We denote the change of the value of $F$ at this step by $\Delta_{s,t}$. By previous discussion, the full gradient $\nabla F$ can be partition into $d(d+1)/2$ block components, such that one of the block $\nabla_{s, t} F$ corresponds to the subproblem. Then
\[
\|\tau_s(\mathbb{U}_s)-\tau_t(\mathbb{U}_t)\| |\Delta_{s,t}| \geq \|\nabla_{s, t} F\|^2.
\]
\end{proposition}
\begin{proof}
Given $1 \le s < t \le d+1$, the sub-step in \eqref{eq:nearest point on flag} is 
\begin{align*}
\min \quad & \lVert \tau_s(\mathbb{U}_s) - \tau_s(\mathbb{W}_s)   \rVert_F^2 + \Vert \tau_t(\mathbb{U}_t) - \tau_t(\mathbb{W}_t)   \rVert_F^2 \nonumber \\
\text{s.t.} \quad & (\mathbb{W}_s,\mathbb{W}_t) \in \Gr(m_s, m_s + m_t) \times \Gr(m_t, m_s + m_t)  \\
\quad & \mathbb{W}_s \perp \mathbb{W}_t  \nonumber
\end{align*}
In particular, $\mathbb{W}_s \oplus \mathbb{W}_t =\left( \bigoplus_{j\ne s,t} \mathbb{W}_j \right)^\perp$ is a fixed $(m_s + m_t)$-dimensional vector space represented by $V_{s, t} := [V_s, V_t]$. We construct the matrix $V^\perp$ whose columns form an orthonormal basis of $\left( \bigoplus_{j\ne s,t} \mathbb{W}_j \right)^\perp$. The choice of $\mathbb{W}_s, \mathbb{W}_t$ can be further specified by an orthogonal matrix $W \in \O(m_s + m_t)$ so that $V_{s, t} W = [W_s, W_t]$ where $W_s, W_t$ span $\mathbb{W}_s, \mathbb{W}_t$ respectively. As a result, the images of $\mathbb{W}_s, \mathbb{W}_t$ can be written as
\[
\tau_s(\mathbb{W}_s) = V_{s, t} W \I_{m_s,m_t} W^\tp V_{s, t}^\tp + V^\perp (V^\perp)^\tp,\quad  \tau_t(\mathbb{W}_t) = -V_{s, t} W \I_{m_s,m_t} W^\tp V_{s, t}^\tp + V^\perp (V^\perp)^\tp.
\]
\eqref{eq:alternating nearest point on flag} follows easily by taking $A_1 = V_{s, t}^\tp \tau_s(\mathbb{U}_s) V_{s, t}$, $A_2 = V_{s, t}^\tp \tau_t(\mathbb{U}_t)V_{s, t}$. 

Next we observe that the objective function in \eqref{eq:alternating nearest point on flag} can further be re-written as 
\begin{align*}
& \lVert A_1 \rVert_F^2 + \lVert A_2 \rVert_F^2 + 2(m_s + m_t) - 2 \langle A_1, W \I_{m_s,m_t} W^\tp \rangle + 2 \langle A_2, W \I_{m_s,m_t} W^\tp \rangle \\
=& \lVert A_1 \rVert_F^2 + \lVert A_2 \rVert_F^2 + 2(m_s + m_t) + 2 \langle A_2 - A_1 , W\I_{m_s,m_t} W^\tp \rangle. 
\end{align*}
Therefore, the problem \eqref{eq:alternating nearest point on flag} is equivalent to 
\begin{align}\label{eq:substep alternating nearest point on flag}
\min \quad & \langle A_2 - A_1 , W\I_{m_s, m_t} W^\tp \rangle  \\
\text{s.t.} \quad & W\in \O(m_s + m_t) \nonumber
\end{align}
By Lemma~\ref{lem:gradient estimate}, we may conclude that a solution to \eqref{eq:substep alternating nearest point on flag} is $W_\ast$, which can be obtained by the SVD of $A_1 - A_2$. Furthermore, we have 
\begin{align*}
\|\nabla_{s, t} F\|^2 \leq \|A_2-A_1\| |\Delta_{s,t}| \leq \|\tau_s(\mathbb{U}_s)-\tau_t(\mathbb{U}_t)\| |\Delta_{s,t}|.
\end{align*}
\end{proof}

\begin{theorem}
Consider a randomized version of Algorithm~\ref{alternating} for problem~\eqref{eq:nearest point on flag}. At each step, choose $(s_i, t_i)$ uniformly from all possible $(s, t)$. Let $\mathbb{W}_i$ be the point at step $i$. Then with probability 1, every cluster point of $\mathbb{W}_i$ is a stationary point. (Because flag manifold is compact, cluster point exists.)
\end{theorem}
\begin{proof}
If $\|\tau_s(\mathbb{U}_s)-\tau_t(\mathbb{U}_t)\| = 0$ for all $s, t$, then the function is trivial and there is nothing to prove. Otherwise, there is a set $A \subseteq \{(s, t) \mid 1\leq s < t \leq d+1\}$ such that $\|\tau_s(\mathbb{U}_s)-\tau_t(\mathbb{U}_t)\| \neq 0$ if and only if $(s, t) \in A$. At each step $i$, assume $\argmax_{(s, t) \in A} \|\nabla_{s, t} F(\mathbb{W}_i)\|$ is achieved for $(s^*, t^*)$. If $(s_i, t_i) = (s^*, t^*)$, then
\begin{align*}
F(\mathbb{W}_i)-F(\mathbb{W}_{i+1}) &\geq \frac{\|\nabla_{s^*, t^*} F(\mathbb{W}_i)\|^2}{\|\tau_{s^*}(\mathbb{U}_{s^*})-\tau_{t^*}(\mathbb{U}_{t^*})\|} \\
&\geq \frac{\max \|\nabla_{s, t} F(\mathbb{W}_i)\|^2}{\max \|\tau_{s}(\mathbb{U}_{s})-\tau_{t}(\mathbb{U}_{t})\|}\\
&\geq C \|\nabla F(\mathbb{W}_i)\|^2,
\end{align*}
where $C$ is a constant independent of $\mathbb{W}_i$. If $(s_i, t_i) \neq (s^*, t^*)$, at least we have $F(\mathbb{W}_i)-F(\mathbb{W}_{i+1}) \geq 0$. So
\[
\mathbb{E} F(\mathbb{W}_i)- \mathbb{E} F(\mathbb{W}_{i+1}) \geq \frac{2C}{n(n-1)} \|\nabla F(\mathbb{W}_i)\|^2.
\]
Summing from $i=0$ to $\infty$, and take expectation, we have
\[
\mathbb{E} [ F(\mathbb{W}_0)-\lim_{i \to \infty} F(\mathbb{W}_{i}) ] \geq C' \mathbb{E} \sum_{i=0}^{\infty} \|\nabla F(\mathbb{W}_i)\|^2.
\]
So with probability 1, $\sum_{i=0}^{\infty} \|\nabla F(\mathbb{W}_i)\|^2$ exists and $\|\nabla F(\mathbb{W}_i)\|$ converges to 0. Any cluster point must be a stationary point.
\end{proof}

\section{Numerical experiments}
In this section, we consider the function
\[
f(V) = \sum_{k=1}^{d} \tr(V_k^\tp A_k V_k),
\]
where $A_i$ is randomly generated symmetric matrix, $V_i$ is the submatrix of $V$ with index $1\leq i\leq n, n_{k-1} < j \leq n_k$, i.e., the basis of $\mathbb{W}_k$. This function is clearly a function on the flag manifold $\Flag(n_1,\dots, n_d;n)$.

We choose $\Flag(5, 5; 200)$ and test five methods: (i) gradient descent method under classical embedding metric; (ii) gradient descent method under modified embedding metric; (iii) coordinate gradient descent method under modified embedding metric; (iv) gradient descent method using the quotient model proposed in Algorithm 1 in \cite{YWL19}; (v) coordinate minimization method under modified metric (Algorithm 1). Figure 1 shows the convergence rate averaged over 10 simulations. We also record the running time to hit $\|\nabla f(V)\| \leq 10^{-5}$, averaged over 10 simulations, as shown in Table 1.

Method (ii) is equivalent to (iv), and their convergence rate and running time are similar. All four descent method has comparable performance while the coordinate minimization outperforms them significantly. 
\begin{table}[]
\centering
\begin{tabular}{l|l}
(i) Classic Descent & 21.428s \\ 
\hline
(ii) Modified Descent & 11.291s \\
\hline
(iii) Coordinate Descent & 25.548s \\
\hline
(iv) Quotient Descent & 10.238s \\ 
\hline
(v) Coordinate Minimization & 0.552s 
\end{tabular}
\caption{Running time to hit $\|\nabla f(V)\| \leq 10^{-5}$ of different methods.}
\label{table1}
\end{table}

\begin{figure}[ht]
\centering
\includegraphics[width=0.75\textwidth]{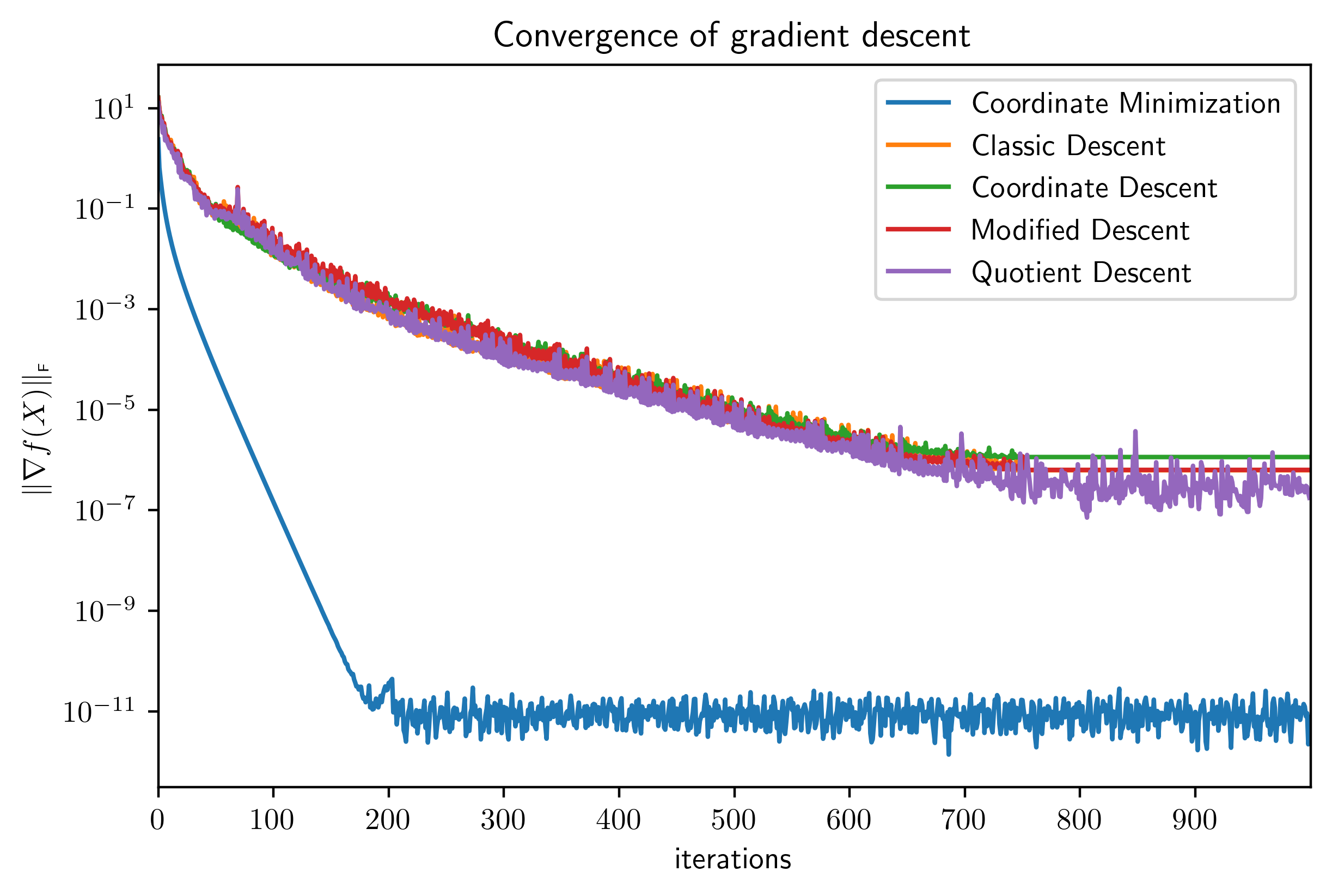}
\caption{Convergence behavior of different methods.}
\label{fig:plot1}
\end{figure}

The coordiante minimization method works best for this special choice of $f(V)$ because the optimization sub-problem has explicit solution and can be solved sufficiently. For more general problems, it might not be the case. However, This special choice of $f$ covers many common problems appeared in flag manifolds optimization. Most notably, it is equivalent to the projection problem under modified embedding \ref{eq:nearest point on flag}. As a result, the extrinsic sample mean problem \cite{BP} for flag manifold under modified embedding can be solved efficiently.


\appendix
\section{Parallel transport with respect to the classical embedding} 
Let $c(t)$ be a curve on $\Flag(n_1,\dots, n_d;n)$ parametrized as
\[
c(t) = V(t) (J_{1},\dots, J_{d}) V(t)^\tp.
\] 
Here $V(t)$ is a curve in $\O(n)$ and hence $\dot{V}(t) = V(t) \Lambda(t)$ for some $\Lambda(t) \in \mathfrak{so}(n)$. If we partition $\Lambda(t)$ as $\Lambda(t) = (\Lambda(j,k))_{j,k=1}^{d+1,d+1}$ with respect to $n = m_1 + \cdots + m_{d+1}$, then $\Lambda(k,k)(t) \equiv 0, k =1,\dots, d+1$.

We notice that by Proposition~\ref{prop:tangentFlag} a vector field $Y(t)$ on $\Flag(n_1,\dots, n_d;n)$ along the curve $c(t)$ can be parametrize as 
\begin{equation}\label{eq:vectorfield}
Y(t) =   V(t) (X(t) J_1 - J_1 X(t),\dots, X(t) J_d - J_d X(t) ) V(t)^\tp,
\end{equation}
where $X(t) = (X(j,k))_{j,k=1}^{d+1,d+1}\in \mathfrak{so}(n)$ is the partition of $X(t)$ with respect to $n = m_1 + \cdots + m_{d+1}$ and $X(k,k)(t) \equiv 0, k =1,\dots, d+1$. We recall that $Y(t)$ is the parallel transport of $Y(0)$ along $c(t)$ if and only if 
\[
\proj^{\mathbb{T}}_{c(t)} (\dot{Y}(t)) = 0.
\]

By differentiating \eqref{eq:vectorfield}, we obtain 
\[
\dot{Y}(t) =  V(t) (\Delta_1(t),\dots, \Delta_d(t)) V^\tp(t),
\]
where 
\begin{align*}
\Delta_k (t) &= (\dot{X}(t) J_k - J_k \dot{X}(t) )  + {\Lambda}(t) (X(t) J_k - J_k X(t)) - (X(t) J_k - J_k X(t)) {\Lambda}(t)  \\
&= (\dot{X}(t) J_k - J_k \dot{X}(t)) + ( {\Lambda}(t) X(t) J_k +  J_k X(t) {\Lambda}(t) ) - ({\Lambda}(t) J_k X(t) + X(t) J_k {\Lambda}(t)).
\end{align*}
Similar to what we have done in Subsection~\ref{subsec:geodesic}, we set
\begin{align}
T_1(t) &= V(t)(\dot{X}(t) J_1 - J_1 \dot{X}(t),\dots, \dot{X}(t) J_d - J_d \dot{X}(t)) V(t)^\tp, \\
T_2(t) &= V(t) ( {\Lambda}(t) X(t) J_1 +  J_1 X(t) {\Lambda}(t),\dots, {\Lambda}(t) X(t) J_d +  J_d X(t) {\Lambda}(t)) V^\tp(t),\\
T_3(t) &= V(t)( {\Lambda}(t) J_1 X(t) + X(t) J_1 {\Lambda}(t) ,\dots,  {\Lambda}(t) J_d X(t) + X(t) J_d {\Lambda}(t)) V^\tp(t),
\end{align}
thus $\dot{Y}(t) =T_1(t) + T_2(t) - T_3(t)$. By definition, we conclude that $T_1(t) \in T_{c(t)}\Flag(n_1,\dots, n_d;n)$ and hence to determine $\proj^{\mathbb{T}}_{c(t)} (\dot{Y}(t))$, it is sufficient to compute $\proj^{\mathbb{T}}_{c(t)} (T_2(t))$ and $\proj^{\mathbb{T}}_{c(t)} (T_3(t))$ respectively.

\begin{lemma}\label{lem:ptT2}
There exists some symmetric matrices $W_1,\dots, W_d$ where 
\[
W_k(p,q) = \begin{cases}
\sum_{s=1}^{d+1}  X(k,s) {\Lambda}(s,q) - \sum_{s=1}^{d+1}  {\Lambda}(s,k)^\tp X(q,s)^\tp,~\text{if}~p=k,q\ne k, \\ 
\sum_{s=1}^{d+1}  {\Lambda}(s,p)^\tp X(k,s)^\tp  - \sum_{s=1}^{d+1}  X(p,s) {\Lambda}(s,k) ,~\text{if}~q=k,p\ne k, \\ 
0,~\text{otherwise}
\end{cases}
\]
such that $\proj^{\mathbb{T}}_{c(t)} (T_2(t)) = V(t) (W_1(t),\dots, W_d(t)) V^\tp(t)$.
\end{lemma}
\begin{proof}
For each $k=1,\dots, d$, we first compute ${\Lambda}(t) X(t) J_k +  J_k X(t) {\Lambda}(t)$. Indeed, since $X(t), \Lambda(t) \in \mathfrak{so}(n)$ and $J_k$ is a diagonal matrix, we have 
\[
J_k X(t) {\Lambda}(t) = \left( {\Lambda}(t) X(t) J_k \right)^\tp.
\]
Therefore we only need to compute ${\Lambda}(t) X(t) J_k $. To do so, we partition $\Lambda(t)$ and $X(t)$ as 
\begin{equation}\label{lem:ptT2:eq1}
\Lambda(t) = (\Lambda(p,q)),\quad X(t) = (X(p,q)),\quad 1\le p,q \le d+1.
\end{equation}
The $(p,q)$-th entry of $J_k X(t) {\Lambda}(t)$ is 
\begin{align*}
\sum_{l,s=1}^{d+1} J_k(p,l) X(l,s) {\Lambda}(s,q) &= \sum_{s=1}^{d+1} J_k (p,p) X(p,s) {\Lambda}(s,q) \\
&= (2\delta_{pk} - 1) \sum_{s=1}^{d+1}  X(p,s) {\Lambda}(s,q).
\end{align*}
Hence the $(p,q)$-th block of ${\Lambda}(t) X(t) J_k +  J_k X(t) {\Lambda}(t)$ is simply 
\begin{align*}
(2\delta_{pk} - 1) \sum_{s=1}^{d+1}  X(p,s) {\Lambda}(s,q) + (2\delta_{qk} - 1) \sum_{s=1}^{d+1}  {\Lambda}(s,p)^\tp X(q,s)^\tp. 
\end{align*}
In particular, for $q \ne k$, the $(k,q)$-th block of ${\Lambda}(t) X(t) J_k +  J_k X(t) {\Lambda}(t)$ is
\begin{align*}
\sum_{s=1}^{d+1}  X(k,s) {\Lambda}(s,q) - \sum_{s=1}^{d+1}  {\Lambda}(s,k)^\tp X(q,s)^\tp
\end{align*}
and if moreover that $q\le d$, then the $(k,q)$-th block of ${\Lambda}(t) X(t) J_q +  J_q X(t) {\Lambda}(t)$ is 
\begin{align*}
- \left( \sum_{s=1}^{d+1}  X(k,s) {\Lambda}(s,q) - \sum_{s=1}^{d+1}  {\Lambda}(s,k)^\tp X(q,s)^\tp \right). 
\end{align*}
This implies that 
\[
\proj^{\mathbb{T}}_{c(t)} (T_2(t)) = V(t) (W_1(t),\dots, W_d(t)) V^\tp(t), 
\]
where 
\[
W_k(p,q) = \begin{cases}
\sum_{s=1}^{d+1}  X(k,s) {\Lambda}(s,q) - \sum_{s=1}^{d+1}  {\Lambda}(s,k)^\tp X(q,s)^\tp,~\text{if}~p=k,q\ne k, \\ 
\sum_{s=1}^{d+1}  {\Lambda}(s,p)^\tp X(k,s)^\tp  - \sum_{s=1}^{d+1}  X(p,s) {\Lambda}(s,k) ,~\text{if}~q=k,p\ne k, \\ 
0,~\text{otherwise}.
\end{cases}
\]
\end{proof}

\begin{lemma}\label{lem:ptT3}
There exists symmetric matrices $Z_1,\dots, Z_d$ where
\[
Z_k(p,q) = \begin{cases}
-\sum_{1 \le l \le d, l\ne k} \left( 
 X(k,l)  {\Lambda}(l,d+1)
+ {\Lambda}(l,k)^\tp X(d+1,l)^\tp 
\right),~\text{if}~p=k,q = d + 1, \\
-\sum_{1 \le l \le d, l\ne k} \left( 
 X(d+1,l)  {\Lambda}(l,k)
+ {\Lambda}(l,d+1)^\tp X(k,l)^\tp 
\right),~\text{if}~p=d+1,q=k, \\
0,~\text{otherwise}.
\end{cases}
\]
such that $\proj^{\mathbb{T}}_{c(t)} (T_3(t)) = V(t) (Z_1(t),\dots, Z_d(t)) V^\tp(t)$.
\end{lemma}
\begin{proof}
We compute ${\Lambda}(t) J_k X(t) + X(t) J_k {\Lambda}(t)$ for each $k=1,\dots, d$. We partition $\Lambda(t)$ and $X(t)$ as in \eqref{lem:ptT2:eq1} respectively. We also notice that 
\[
X(t) J_k {\Lambda}(t) = \left( {\Lambda}(t) J_k X(t) \right)^\tp
\]
so that it is sufficient to compute $X(t) J_k {\Lambda}(t)$. The $(p,q)$-th block of $X(t) J_k {\Lambda}(t)$ is 
\begin{align*}
\sum_{l,s=1}^{d+1} X(p,l) J_k(l,s) {\Lambda}(s,q) &=\sum_{l=1}^{d+1} X(p,l) J_k(l,l) {\Lambda}(l,q) \\
&=\sum_{l=1}^{d+1} (2\delta_{kl} - 1) X(p,l)  {\Lambda}(l,q)
\end{align*}
Hence the $(p,q)$-th block of ${\Lambda}(t) J_k X(t) + X(t) J_k {\Lambda}(t)$ is 
\[
\sum_{1 \le l \le d+1, l \ne p,q} (2\delta_{kl} - 1)\left( 
 X(p,l)  {\Lambda}(l,q)
+ {\Lambda}(l,p)^\tp X(q,l)^\tp 
\right).
\]
In particular, for $q\ne k$, the $(k,q)$-th block of ${\Lambda}(t) J_k X(t) + X(t) J_k {\Lambda}(t)$ is 
\[
-\sum_{1 \le l \le d+1, l\ne k,q} \left( 
 X(k,l)  {\Lambda}(l,q)
+ {\Lambda}(l,k)^\tp X(q,l)^\tp 
\right)
\]
which is the same as the $(k,q)$-th block of ${\Lambda}(t) J_q X(t) + X(t) J_q {\Lambda}(t)$ if moreover $q \le d$. Hence we have 
\[
\proj^{\mathbb{T}}_{c(t)} (T_3(t)) = V(t) (Z_1(t),\dots, Z_d(t)) V(t)^\tp,
\]
where for each $1\le k \le d$, 
\[
Z_k(p,q) = \begin{cases}
-\sum_{1 \le l \le d, l\ne k} \left( 
 X(k,l)  {\Lambda}(l,d+1)
+ {\Lambda}(l,k)^\tp X(d+1,l)^\tp 
\right),~\text{if}~p=k,q = d + 1, \\
-\sum_{1 \le l \le d, l\ne k} \left( 
 X(d+1,l)  {\Lambda}(l,k)
+ {\Lambda}(l,d+1)^\tp X(k,l)^\tp 
\right),~\text{if}~p=d+1,q=k, \\
0,~\text{otherwise}.
\end{cases}
\]
\end{proof}

\begin{proposition}[parallel transport along any curve]\label{prop:ptflag}
Let $c(t) = V(t) (J_1,\dots, J_d) V(t)^\tp$ be a curve on $\Flag(n_1,\dots,n_d;n)$ and let $Y(t)$ be a vector field along the curve $c(t)$, parametrized as in \eqref{eq:vectorfield}. Then $Y(t)$ is a parallel transport if and only for each pair $(k,q)$ such that $1 \le k < q \le d+1$, we have  
\footnotesize{
\begin{equation}\label{prop:ptflag:eq1}
-2 \dot{X}(k,q) + \sum_{\substack{1\le s \le d+1 \\ s\ne k,q}} \left( X(k,s) {\Lambda}(s,q) - {\Lambda}(k,s) X(s,q) \right) + \delta_{q,d+1} \sum_{\substack{1 \le l \le d \\ l\ne k}} \left( 
 X(k,l)  {\Lambda}(l,d+1)
+ {\Lambda}(k,l) X(l,d+1) 
\right) = 0,
\end{equation}}
\normalsize
which can be written in a more compact form:
\begin{equation}\label{prop:ptflag:eq2}
2 \dot{X} =\pi\left( [X, {\Lambda}]\right) + \begin{bmatrix}
0 & X_0 {\Lambda}_1 + {\Lambda}_0 X_1 \\
-(X_0 {\Lambda}_1 + {\Lambda}_0 X_1)^\tp & 0
\end{bmatrix} ,
\end{equation}
where $X_0,\Lambda_0\in \mathbb{R}^{(n-n_{d+1})\times (n-n_{d+1})} ,X_1,\Lambda_1\in \mathbb{R}^{(n-n_{d+1})\times n_{d+1}}$ are submatrices determined by partitions
\[
X = \begin{bmatrix}
X_0 & X_1 \\ 
-X_1^\tp & 0 
\end{bmatrix},\quad \Lambda = \begin{bmatrix}
\Lambda_0 & \Lambda_1 \\ 
-\Lambda_1^\tp & 0 
\end{bmatrix},
\]
and $\pi(A)$ is defined by setting all diagonal blocks of $A\in \mathfrak{so}(n)$ equal to zero. 
\end{proposition}
Before we proceed, we remark that in particular if $X = {\Lambda}$, then $[X,{\Lambda}] = 0$ and \eqref{prop:ptflag:eq2} reduces to the geodesic equation \eqref{prop:gedoesic:eq1}.

Next we re-write each term of \eqref{prop:ptflag:eq1} using \eqref{eq:vectorization}. This leads to 
\begin{align*}
\vect \left(  X(k,s) {\Lambda}(s,q) \right) &= ({\Lambda}(s,q)^\tp \otimes \I_{m_k}) \vect ( X(k,s)) =   -({\Lambda}(q,s) \otimes \I_{m_k}) \vect ( X(k,s)) \\
\vect \left( {\Lambda}(s,k)^\tp X(q,s)^\tp \right) &= (\I_{m_q} \otimes {\Lambda}(s,k)^\tp) \vect (X(q,s)^\tp) = (\I_{m_q} \otimes {\Lambda}(k,s)) \vect (X(s,q)) 
\end{align*}
and hence \eqref{prop:ptflag:eq1} becomes  
\small
\begin{align*}
 2\vect \left( \dot{X}(k,q) \right) &= - \sum_{\substack{1 \le s \le d+1 \\ s \ne k, q}} \left(  
  (\I_{m_q} \otimes {\Lambda}(k,s)) \vect (X(s,q)) + ({\Lambda}(q,s) \otimes \I_{m_k}) \vect ( X(k,s)) 
  \right) \\
&+ \delta_{q,d+1} \sum_{\substack{1 \le s \le d \\ s\ne k}} \left( 
(\I_{m_{d+1}} \otimes {\Lambda}(k,s) ) \vect (X(s,d+1)) -({\Lambda}(d+1,s) \otimes \I_{m_k}) \vect ( X(k,s)) 
\right),
\end{align*}\normalsize
which using relations $X(p,q) = -X(q,p)^\tp$ and $\vect(A^\tp) = K^{(m,n)} \vect (A)$\footnote{$K^{(m,n)}$ is the commutation matrix defined in \eqref{eq:commutation}.}, can be written  as 
\begin{equation}\label{eq:pt1}
\begin{bmatrix}
\vect (\dot{X}(1,2)) \\
\vect (\dot{X}(1,3)) \\
\cdots \\
\vect (\dot{X}(d-1,d))
\end{bmatrix} = \Phi(t) \begin{bmatrix}
\vect ({X}(1,2)) \\
\vect ({X}(1,3)) \\
\cdots \\
\vect ({X}(d,d+1))
\end{bmatrix} 
\end{equation}
for some $\binom{n}{2} \times \binom{n}{2}$ matrix function $\Phi(t)$. Now according to Theorem~\ref{thm:PBseries}, \eqref{eq:pt1} can be solved by Peano--Baker series associated to the coefficient matrix $\Phi(t)$. 

We again take $d = 2$ for example. In this case, we write 
\[
\Lambda (t) = \begin{bmatrix}
0 & A(t) & B(t) \\
-A^\tp (t) & 0 & C(t) \\
-B^\tp (t) & -C^\tp(t) & 0 
\end{bmatrix},\quad A(t)\in \mathbb{R}^{m_1 \times m_2},B(t)\in \mathbb{R}^{m_1 \times m_3},C(t)\in \mathbb{R}^{m_2 \times m_3},
\]
\[
X (t) = \begin{bmatrix}
0 & U(t) & V(t) \\
-U^\tp (t) & 0 & W(t) \\
-V^\tp (t) & -W^\tp(t) & 0 
\end{bmatrix},\quad U(t)\in \mathbb{R}^{m_1 \times m_2},V(t)\in \mathbb{R}^{m_1 \times m_3},W(t)\in \mathbb{R}^{m_2 \times m_3}.
\]
Then we have 
\scriptsize{
\[
\proj^{\mathbb{T}}_{c(t)} ( T_1(t) ) = 2  V(t)  \left( 
\begin{bmatrix}
0 & -\dot{U}(t) & -\dot{V}(t) \\
-\dot{U}^\tp(t) & 0 & 0  \\
-\dot{V}^\tp(t) & 0 & 0
\end{bmatrix}, 
\begin{bmatrix}
0 & \dot{U}(t) & 0 \\
\dot{U}^\tp(t) & 0 & -\dot{W}(t) \\
0 & -\dot{W}^\tp(t) & 0
\end{bmatrix} 
\right) V(t)^\tp, 
\]
\[
\begin{split}
\proj^{\mathbb{T}}_{c(t)} ( T_3 (t) ) = 
V(t) \left( 
\begin{bmatrix}
0 & 0 & -({A}(t) W(t) + U(t) {C}(t) ) \\
0& 0 & 0  \\
-({A}(t) W(t) + U(t) {C}(t) )^\tp & 0 & 0
\end{bmatrix}, \right. \\
\left.
\begin{bmatrix}
0 & 0 & 0\\
0 & 0 & {A}(t)^\tp V(t) + U(t)^\tp {B}(t) \\
0 & V(t)^\tp {A}(t)  +  {B}(t)^\tp U(t)  & 0
\end{bmatrix} 
\right)  V(t)^\tp.
\end{split}
\]
\[
\begin{split}
\proj^{\mathbb{T}}_{c(t)} ( T_2 (t) ) =   V(t) 
\left( 
\begin{bmatrix}
0 & {B}(t) W(t)^\tp - V(t) {C}(t)^\tp & -{A}(t) W(t) + U(t) {C}(t) \\
W(t) {B}(t)^\tp - {C}(t) V(t)^\tp & 0 & 0  \\
-W(t)^\tp {A}(t)^\tp  + {C}(t) ^\tp U(t)^\tp  & 0 & 0
\end{bmatrix},   \right. \\
\left.
\begin{bmatrix}
0 & -{B}(t) W(t)^\tp + V(t) {C}(t)^\tp & 0\\
-W(t) {B}(t)^\tp + {C}(t) V(t)^\tp & 0 & {A}(t)^\tp V(t) - U(t)^\tp {B}(t)  \\
0 & V(t)^\tp {A}(t) -  {B}(t)^\tp U(t)  & 0
\end{bmatrix} 
\right)  V(t)^\tp.
\end{split}
\]
}
\normalsize
Hence the system for $X(t)$ to be a parallel transport is given by 
\begin{align*}
2\dot{U}(t) &=-  V(t) {C}(t)^\tp +  {B}(t) W(t)^\tp  = (-{C}(t) \otimes \I_{m_1} )\vect (V(t))   + (\I_{m_2} \otimes {B}(t))K^{(m_2,m_3)} \vect (W(t)) , \\
\dot{V}(t) &= U(t) {C}(t) = ({C}(t)^\tp \otimes \I_{m_1}) \vect(U(t)), \\ 
\dot{W}(t) &= -U(t)^\tp {B}(t) = - ({B}(t)^\tp \otimes \I_{m_2}) \vect(U(t)).
\end{align*}
Hence we have 
\[
\begin{bmatrix}
\vect(\dot{U}(t)) \\
\vect(\dot{V}(t)) \\
\vect(\dot{W}(t)) \\
\end{bmatrix} = \begin{bmatrix}
0 & -\frac{1}{2} {C}(t) \otimes \I_{m_1} & \frac{1}{2} ( \I_{m_2} \otimes {B}(t))K^{(m_2,m_3)} \\
{C}(t)^\tp \otimes \I_{m_1} & 0 & 0 \\
-{B}(t) \otimes \I_{m_2} & 0 & 0
\end{bmatrix}
\begin{bmatrix}
\vect(U(t)) \\
\vect(V(t)) \\
\vect(W(t)) \\
\end{bmatrix}.
\]
\section{Parallel transport with respect to the modified embedding} 
Now we proceed to discuss the parallel transport of a tangent vector along a curve on a flag manifold. Again we parametrize a curve $c(t)$ on $\Flag(n_1,\dots, n_d;n)$ as \eqref{eq:newmodelcurve}. Let $Y(t)$ be a vector field on $\Flag(n_1,\dots, n_d;n)$ along the curve $c(t)$. Then we may correspondingly parametrize $Y(t)$ as 
\begin{equation}\label{eq:newmodelvectorfield}
Y(t) =  V (t) (X(t) J_1 - J_1 X(t),\dots, X(t) J_{d+1} - J_{d+1} X(t) ) V(t)^\tp,
\end{equation}
where $X(t) = (X(j,k))_{j,k=1}^{d+1,d+1}\in \mathfrak{so}(n)$ is the partition of $X(t)$ with respect to $n = m_1 + \cdots + m_{d+1}$ and $X(k,k)(t) \equiv 0, k =1,\dots, d+1$. We notice that $\dot{Y}(t) = T_1(t) + T_2(t) - T_3(t)$ where 
\begin{align}
T_1(t) &= V(t) (\dot{X}(t) J_1 - J_1 \dot{X}(t),\dots, \dot{X}(t) J_{d+1} - J_{d+1} \dot{X}(t)) V^\tp(t), \\
T_2(t) &= V(t) ( {\Lambda}(t) X(t) J_1 +  J_1 X(t) {\Lambda}(t),\dots, {\Lambda}(t) X(t) J_{d+1} +  J_{d+1} X(t) {\Lambda}(t)) V^\tp(t),\\
T_3(t) &= V(t) ( {\Lambda}(t) J_1 X(t) + X(t) J_1 {\Lambda}(t) ,\dots,  {\Lambda}(t) J_{d+1} X(t) + X(t) J_{d+1} {\Lambda}(t)) V^\tp(t).
\end{align}

Similar computations in proofs of Lemmas~\ref{lem:ptT2} and \ref{lem:ptT3} lead to the following 
\begin{lemma}\label{lem:newmodelptTi}
Let $c(t), \Lambda(t), Y(t), X(t), T_1(t), T_2(t), T_3(t)$ be as above. We have 
\begin{enumerate}
\item $T_1(t) \in \mathbb{T}_{c(t)} \Flag(n_1,\dots, n_d;n)$.
\item
$\proj_{c(t)}^{\mathbb{T}} (T_2(t)) = V(t) (W_1(t),\dots, W_{d+1}(t)) V^\tp(t)$, where 
\[
W_k(p,q) = \begin{cases}
\sum_{s=1}^{d+1}  X(k,s) {\Lambda}(s,q) - \sum_{s=1}^{d+1}  {\Lambda}(s,k)^\tp X(q,s)^\tp,~\text{if}~p=k,q\ne k, \\ 
\sum_{s=1}^{d+1}  {\Lambda}(s,p)^\tp X(k,s)^\tp  - \sum_{s=1}^{d+1}  X(p,s) {\Lambda}(s,k) ,~\text{if}~q=k,p\ne k, \\ 
0,~\text{otherwise}.
\end{cases}
\]
\item $\proj_{c(t)}^{\mathbb{T}} (T_3(t)) = 0$.
\end{enumerate}
\end{lemma}
Now we recall that $Y(t)$ is a parallel transport along the curve $c(t)$ if and only if $\proj_{c(t)}(\dot{Y}(t)) = 0$. Combining this with Lemma~\ref{lem:newmodelptTi}, we may derive the equation for parallel transport. 
\begin{proposition}[Parallel transport along any curve]\label{prop:newmodelpt}
Let $c(t), \Lambda(t), Y(t), X(t)$ be as above. The vector field $Y(t)$ along $c(t)$ is a parallel transport if and only if 
\begin{equation}\label{eq:prop:newmodelpt}
\dot{X}(t) = \frac{1}{2} \pi \left([X(t) , {\Lambda}(t)] \right),
\end{equation}
where $\pi(A)$ is the matrix obtained by setting all diagonal blocks equal to zero for $A\in \mathfrak{so}(n)$.
\end{proposition}

In particular, if $c(t) = V(t) (J_1,\dots, J_{d+1}) V^\tp(t)$ is a geodesic, then Proposition~\ref{prop:newmodelgedoesic} implies $\dot{V}(t) = V(t) A$ for some constant matrix $A\in \mathfrak{so}(n)$, from which we obtain the following characterization of a parallel transport along a geodesic in $\Flag(n_1,\dots, n_{d};n)$.
\begin{corollary}[Parallel transport along a geodesic]
If $c(t) = V(t) (J_1,\dots, J_{d+1}) V^\tp(t)$ is the geodesic curve passing through $V(0) (J_1,\dots, J_{d+1}) V^\tp(0)$ with the tangent direction 
\[
V(0) (A J_1 - J_1A, \dots, A J_{d+1} - J_{d+1}A) V^\tp(0),\quad A\in \mathfrak{so}(n),A(k,k) = 0, k=1,\dots, d+1,
\]
then the parallel transport $Y(t)$ of 
\[
Y(0) =V(0)  (BJ_1- J_1B,\dots, BJ_{d+1}- J_{d+1}B) V^\tp(0),\quad B\in \mathfrak{so}(n),B(k,k) = 0, k=1,\dots, d+1,
\]
is 
\[
Y(t) =V(t) (X(t)J_1- J_1X(t),\dots, X(t)J_{d+1}- J_{d+1}X(t))  V^\tp(t),
\]
where 
\[
\dot{X}(t) = \frac{1}{2} \pi \left( [X(t), A] \right),\quad X(0) = B,
\]
where $\pi(A)$ is the matrix obtained by setting all diagonal blocks equal to zero for $A\in \mathfrak{so}(n)$.
\end{corollary}

\end{document}